\documentstyle[amssymb,amsfonts]{amsart}

\def\be#1{\begin{equation}\label{#1}}
\def\bas{\begin{align*}}
\def\eas{\end{align*}}
\def\bi{\begin{itemize}}
\def\ei{\end{itemize}}

\newenvironment{proof}{\noindent {\bf Proof} }{\endprf\par}
\def \endprf{\hfill  {\vrule height6pt width6pt depth0pt}\medskip}
\def\emph#1{{\it #1}}
\def\textbf#1{{\bf #1}}

\theoremstyle{plain}
   \newtheorem{theorem}[subsection]{Theorem}

   \newtheorem{proposition}[subsection]{Proposition}
   
   \newtheorem{lemma}[subsection]{Lemma}

\theoremstyle{remark}

\theoremstyle{definition}

\begin{document}

\author{James Wright}
\address{Maxwell Institute of Mathematical Sciences and the School of Mathematics, University of
Edinburgh, JCMB, King's Buildings, Mayfield Road, Edinburgh EH9
3JZ, Scotland} \email{J.R.Wright@@ed.ac.uk}




\title[Exponential sums and polynomial congruences in two variables]
{Exponential sums and polynomial congruences in two variables: the
quasi-homogeneous case}

\begin{abstract}
We adapt ideas of Phong, Stein and Sturm and ideas
of Ikromov and M\"uller 
from the continuous setting
to various discrete settings, obtaining sharp bounds for exponential
sums and the number of solutions to polynomial
congruences for general quasi-homogeneous polynomials in two
variables. This extends work of Denef and Sperber
and also Cluckers
regarding a conjecture of Igusa in
the two dimensional setting by no longer requiring the polynomial
to be nondegenerate with respect to its Newton diagram.

\end{abstract}

\maketitle
\section{Introduction}
Let $f\in {\Bbb Z}[X,Y]$ be a quasi-homogeneous polynomial in two
variables by which we mean there exist two positive numbers
$\kappa_1, \kappa_2 > 0$ so that $f(r^{\kappa_1}x, r^{\kappa_2} y)
= r f(x,y)$ for every $r\ge 0$.
Our goal is to give sharp uniform bounds on the exponential sums
$$
{\mathcal S}(f;p^s) \ = \ \frac{1}{p^{2s}} \sum_{x \, {\rm mod} \,
p^s}\sum_{y \, {\rm mod} \, p^s} e^{2\pi i f(x,y)/p^s}
$$
where the parameter $p^s$ is a power of a prime number $p$. We
will also obtain precise estimates for the number of solutions to
the polynomial congruence $f(x,y) = 0$ mod $p^s$.


We will be particularly interested in estimates
of the form
\begin{equation}\label{basic-sum}
|{\mathcal S}(f;p^s)| \ \le \ C s^{i(f)} p^{-s/h(f)}
\end{equation}
and such uniform estimates
will be proved for almost every prime $p$ 
where $C$ is an absolute constant depending only on the
degree of $f$; that is, there is an exceptional finite set of primes ${\mathcal
P}(f)$ and a constant $C=C_{{\rm deg}(f)}$
such that \eqref{basic-sum} holds for every
$p\notin {\mathcal P}(f)$. In fact in almost every case the
exponents $h(f)$ and $i(f)$ will be the same
as those arising in the best uniform estimates
for the corresponding euclidean oscillatory integrals
\begin{equation}\label{osc-euclid}
\Bigl| \int\!\!\!\int_{{\Bbb R}^2} e^{2\pi i \lambda f(x,y)} \phi(x,y) dx dy
\Bigr| \ \le \ C \, \bigl[\log(|\lambda|)\bigr]^{i(f)} |\lambda|^{-1/h(f)}
\end{equation}
where the {\it height} of $f$ is defined as $h(f) :=
\sup_z \{d_z(f)\}$, the supremum being taken over all
smooth local coordinate systems $z=(x,y)$ of the origin
and $d_z$ denotes the {\it Newton distance} of $f$ in the
coordinates $z$. See Section \ref{prelim} for precise definitions
of these notions as well as an explicit, intrinsic description of the
height $h(f)$ when $f$ is a quasi-homogeneous polynomial
(if $f(x,y) = ax^j y^k$ is single monomial,
then $h(f) = \max(j,k)$ and when $f$ consists of more than one monomial, the homogeneity dilation parameters $\kappa_1$, $\kappa_2$ are
uniquely determined by $f$; in this case
$h(f)$ can be described explicitly in terms of
$\kappa_1$, $\kappa_2$
and the maximum multiplicity of the real roots of $f$).
The exponent $i(f)$ is sometimes referred to as Varchenko's exponent
or the {\it multiplicity of oscillation} of $f$ and takes
only the values $0$ or $1$; it is always equal to $0$ except
when $h(f)\ge 2$ and the {\it principal face}
of $f$ in {\it adapted coordinates}\footnote{a local coordinate
system $z$ where the supremum defining the height is
achieved; that is, $h(f) = d_z$} is a vertex of the
Newton diagram in which case we set $i(f) = 1$.
Again, in Section \ref{prelim} we
will give precise definitions of these notions
and an explicit, intrinsic description of the
exponent $i(f)$. The estimate \eqref{osc-euclid} is
sharp in the sense that
\begin{equation}\label{osc-sharp}
\lim_{\lambda \to +\infty} \frac{\lambda^{1/h(f)}}{\log^{i(f)}(\lambda)}
\int\!\!\!\int_{{\Bbb R}^2} e^{2\pi i \lambda f(x,y)} \phi(x,y) dx dy \ =
\ c \, \phi(0,0)
\end{equation}
for some nonzero constant $c$ if the support of $\phi$ is sufficiently
small and if the principal face of $f$ in adapted coordinates
is a compact set. For proofs of \eqref{osc-euclid} and \eqref{osc-sharp},
see for example \cite{IM-2} where these results are established for any
smooth real-valued phase $f$ of finite-type.

It turns out that the uniform estimates in \eqref{basic-sum},
discrete analogues of \eqref{osc-euclid},
hold
for every quasi-homogeneous polynomial $f \in {\Bbb Z}[X,Y]$
{\it except} for a
single family of degenerate $f$ of the form
$$
f(x,y) \ \ = \ \ a(b y^2 + cxy + d x^2)^m
$$
where $m\ge 2$ and the quadratic polynomial $b y^2 + c xy + d x^2$ is
irreducible over the rationals ${\Bbb Q}$. In this case
\eqref{basic-sum} holds with the same decay parameter $h(f)$ but
now the $0-1$ valued exponent $i(f) = i_p(f)$ depends on the
prime $p$. For example when $f(x,y) = a(y^2 - 2 x^2)^m$
and $m\ge 2$, it turns
out that $i_p(f) = 1$ when $p \equiv 1$ or $7$ mod $8$ and $i_p(f)
= 0$ when $p \equiv 3$ or $5$ mod $8$.

We denote by $E_m$ the class of functions
$f(x,y) = a(b y^2 + cxy + d x^2)^m$ with
$b y^2 + c xy + d x^2$ irreducible over ${\Bbb Q}$.
Hence when $m\ge 2$, $E_m$ is an exceptional class where
the direct discrete analogue of the euclidean
estimate \eqref{osc-euclid} does not hold. When we turn
to counting solutions of polynomial congruences,
$E_m$ will be an exceptional class
for all $m\ge 1$.
An explanation
of how the classes $E_m$ arise
is given before the statement of Theorem \ref{main-abstract} below.


We also obtain
a version of \eqref{osc-sharp} for ${\mathcal S}(f;p^s)$
in the following theorem.


\begin{theorem}\label{main-sum} For any quasi-homogeneous polynomial
$f\in {\Bbb Z}[X,Y]$,
there is a finite
collection ${\mathcal P}(f)$ of prime numbers and constants
$c, C >
0$, depending only on the degree of $f$, so that
for every prime $p \notin {\mathcal P}(f)$
and $f$ not in any exceptional class $E_m$ with $m\ge 2$,
$$
|{\mathcal S}(f;p^s)| \ \le \ C s^{i(f)} p^{-s/h(f)}
$$
holds and if furthermore $f(x) \not= ax + by$,
\begin{equation}\label{sum-infinite}
c s^{i(f)} p^{-s/h(f)} \ \le \ |{\mathcal S}(f; p^s)|
\end{equation}
holds for infinitely many $s\ge 1$.

When $f$ belongs to some $E_m, \ m\ge 2$,
the above estimates still hold but now $i(f) = i_p(f)$ depends on $p$; more precisely $i_p(f) = 1$ or
$0$ depending on whether the roots of $f$ (a conjugate pair of
algebraic numbers of degree 2 over ${\Bbb Q}$) lie in the $p$-adic
field ${\Bbb Q}_p$ or not, respectively.

\end{theorem}

When $f(x) = ax + by$ is linear such that either
$p^s \not| \, a$ or $p^s \not| \, b$, then ${\mathcal S}(f;p^s) = 0$
and so no lower bound as in \eqref{sum-infinite} holds
in this case. The finite collection ${\mathcal P}(f)$
of exceptional primes which we will work with is a
specific collection which we will describe
precisely in Section \ref{exceptional} below.
The sequence of $s\ge 1$ where the lower
bound \eqref{sum-infinite} holds can be chosen to depend
only on $f$; more precisely, if $f$ is not a single
monomial, then there is a pair $1\le t\le r$
of relatively prime integers, determined by the
dilation parameters $\kappa_1, \kappa_2 > 0$, such
that $f(x,y) = a x^{\alpha} y^{\beta} Q(x^r, y^t)$
for some homogeneous polynomial $Q(u,v)$ of degree $n\ge 1$;
see Section \ref{prelim} below. Then
\eqref{sum-infinite} holds for example for any $s\ge 1$
satisfying $s \equiv 0$ mod $N M_1 M_2 P$ where
$N = t\alpha + r\beta + rtn$, $M_1 = {\rm gcd}(\beta,\alpha + rn)$,
$M_2 = {\rm gcd}(\alpha,\beta + tn)$ and $P$ equals
to the product of the multiplicities of the roots of $Q$.

For quasi-homogeneous polynomials $f\in {\Bbb Z}[X_1,\ldots,X_n]$
in arbitrary number of variables, Denef and Sperber \cite{DS}
and Cluckers \cite{C},
\cite{C-2} have established the estimate \eqref{basic-sum}
when\footnote{there is one exceptional case here; when
$f(x,y) = a x (y - \zeta x^r)$ (or the symmetric example
with $x$ and $y$ interchanged), $h(f)=1$ and so $i(f) = 0$
but the bisectrix passes through the vertex $(1,1)$ and so
the estimates in \cite{DS} or \cite{C-2}, strictly speaking,
carry a linear factor
of $s$.} $f$ is nondegenerate with respect to its Newton diagram which
is related to certain conjectures of Igusa found in \cite{I}
(we remark that any $f$ in an exceptional class $E_m$
for some $m\ge 2$ is degenerate with respect to its Newton diagram).
The estimates in Theorem \ref{main-sum} extend their work in
the two variable setting to arbitrary quasi-homogeneous polynomials.
In fact in \cite{DS}, Denef and Sperber make a conjecture for
general homogeneous polynomials (extended to quasi-homogeneous
polynomials by Cluckers) and Theorem \ref{main-sum} verifies this
conjecture in the two variable setting.
The lower bound
\eqref{sum-infinite} shows the general sharpness of the estimate
with respect to $p$ and $s$.
Sharp estimates for arbitrary quasi-homogeneous polynomials
have been obtained previously by Cluckers \cite{C-3}
in the case when $s=1$ or $s=2$, again for polynomials in any
number of variables.

We turn our attention now to polynomial congruences. Whenever a pair
of integers $(x,y)$ satisfies the congruence $f(x,y) \equiv 0$ mod
$n$, then so does $(x+jn, y+kn)$ for any $(j,k)\in {\Bbb Z}^2$.
Therefore a solution to the congruence $f\equiv 0$ mod $n$ is
defined to be an element in the ring ${\Bbb Z}/n{\Bbb Z} \times
{\Bbb Z}/n{\Bbb Z}$ and if $\# \{f\equiv 0 \, {\rm mod} \, n\}$
denotes the total number of solutions, we will examine the
normalised number of solutions
$$
{\mathcal N}(f;n) \ := \ n^{-2} \# \{ f \equiv 0 \, {\rm mod} \,
n\}.
$$
The quantity ${\mathcal N}(f; n)$ is a multiplicative function of
$n$ and so matters are reduced to studying ${\mathcal N}(f; p^s)$
for powers of a fixed prime $p$. Not surprisingly we obtain similar
estimates for ${\mathcal N}(f;p^s)$ which are direct analogues
of ones arising in euclidean sublevel set estimates which we will
not write down explicitly. In the euclidean situation the decay
parameter $h(f)$ remains the same but Varchenko's exponent needs
slight modification: we define $\nu(f) = 0$ in every case except
when the principal face of $f$ in adapted coordinates is a vertex
of the Newton diagram in which case we set $\nu(f) = 1$. So
the only difference between $i(f)$ and $\nu(f)$ occurs when $h(f) < 2$.

\begin{theorem}\label{main-cong}
For any quasi-homogeneous polynomial $f\in {\Bbb Z}[X,Y]$,
there is a finite collection of prime numbers ${\mathcal P}(f)$
and constants $C,c,c' >0$, depending only on the degree of $f$,
so that for any $p\notin {\mathcal P}(f)$
and $f \notin E_m$ for any $m\ge 1$,
\begin{equation}\label{basic-cong}
c s^{\nu(f)} p^{-s/h(f)} p^{-2} \ \le \ {\mathcal N}(f;p^s) \ \le \
C s^{\nu(f)} p^{-s/h(f)}
\end{equation}
holds and
\begin{equation}\label{cong-infinite}
c' s^{\nu(f)} p^{-s/h(f)} \ \le \ {\mathcal N}(f; p^s)
\end{equation}
holds for infinitely many $s\ge 1$.

When $f$ lies in some $E_m$ with $m\ge 1$, the estimates
\eqref{basic-cong} and \eqref{cong-infinite} still hold but
the exponent $\nu(f)=\nu_p(f)$ now depends on $p$;
more precisely $\nu_p(f) = 1$ or
$0$ depending on whether the roots of $f$ (a conjugate pair of
algebraic numbers of degree 2 over ${\Bbb Q}$) lie in the $p$-adic
field ${\Bbb Q}_p$ or not, respectively.
\end{theorem}

Simple examples show that the factor $p^{-2}$ in the lower
bound in \eqref{basic-cong} cannot be replaced by $p^{-1}$;
for instance if
$f(x,y) = y^4 - 2 x^6$,
then
$h(f) = 12/5$, $\nu(f) = 0$ and the analysis
in Section \ref{homogeneous} shows that
${\mathcal N}(f;p^s) \le c p^{-5s/12} p^{-19/12}$ if
$s \equiv 1$ mod 12 and either $p\equiv$ 3 or 5
mod 8. See Section \ref{I+II+III} for details.
Nevertheless there is a natural large class of quasi-homogeneous
polynomials where the factor $p^{-2}$ can be replaced
by $p^{-1}$; see the comments after the statement of
Theorem \ref{main-abstract} below.
When $f(x) = ax + by$ is linear, the situation
of polynomial congruences differs from the situation
of exponential sums. In this case, ${\mathcal N}(f;p^s) =
p^{-s}$ if either $p \not| \, a$ or $p\not| \, b$,
$h(f) = 1$ and $\nu(f) = 0$ so that
\eqref{basic-cong} and \eqref{cong-infinite} hold.

As we will see the proofs of Theorems \ref{main-sum} and
\ref{main-cong}
are very elementary, relying on a sharp structural statement
for the solution set of general polynomial congruences
of a single variable
found in \cite{W-1}; see also \cite{W-2}. This result
is a nonarchimedean
version of a result of Phong, Stein and Sturm \cite{PSS}
about polynomial sublevel sets in euclidean spaces.
The result in \cite{W-1} is valid in general
settings and Theorems \ref{main-sum} and \ref{main-cong}
generalise accordingly.

Let ${\mathfrak o}$ be any ring endowed with a nontrivial discrete
valuation $|\cdot|$ so that $|x|\le 1$ for every $x\in {\mathfrak
o}$. Let us suppose that the nonzero prime ideal
${\mathfrak p} := \{x\in {\mathfrak o}: |x|<1\}$ is maximal
in ${\mathfrak o}$ such that the localisation of ${\mathfrak o}$
to ${\mathfrak p}$
is the valuation ring
$\{x\in K : |x| \le 1 \}$
of the field of fractions
$K$ of ${\mathfrak o}$ induced by $|\cdot|$.
The valuation ring
has a unique maximal ideal generated by a prime element $\pi$
which may assume lies in ${\mathfrak o}$.
We make the finiteness assumption that the residue class
field ${{\mathfrak o}}/{{\mathfrak p}}$ is finite, say
with $q=p^f$ elements where $p$ is prime, and we normalise
the valuation so that $|\pi| = q^{-1}$.

The maximality of ${\mathfrak p}$ implies that the fields
${\mathfrak o}/{\mathfrak p} \simeq  {\bar{\mathfrak o}}/\pi
{\bar{\mathfrak o}}$
are isomorphic where
${\bar {\mathfrak o}}$ denotes the completion of ${\mathfrak o}$
with respect to $|\cdot|$.
Furthermore the field of
fractions of ${\bar{\mathfrak o}}$ is ${\bar K}$, the
completion of $K$ with respect to $|\cdot|$, and
the valuation extends uniquely to ${\bar K}$.
Finally ${\bar{\mathfrak o}}$ is the valuation ring
of ${\bar K}$ with respect to $|\cdot|$; that is,
${\bar{\mathfrak o}} = \{x\in {\bar K}: |x|\le 1 \}$.

Our finiteness hypothesis on the residue class field
implies that ${\bar K}$ is a local field. Hence ${\bar K}$ is a
finite field extension of the $p$-adic field ${\mathbb Q}_p$ (in
the characteristic 0 case) or the field ${\mathbb F}_p((\pi))$ of
Laurent series with coefficients in the field ${\mathbb F}_p$ of
integers modulo $p$ (in the positive characteristic case); in the
latter case we can be more explicit, namely ${\bar K} = {\mathbb
F}_q ((\pi))$ where $q = p^f$ is defined above as the number of
elements in the residue class field. If $n$ is the degree of
${\bar K}$ over ${\mathbb Q}_p$ or ${\mathbb F}_p((\pi))$, then $n
= e f$ where $f$, defined above, is the residual degree and the
exponent $e$ is the ramification index of this extension. In the
characteristic 0 case, viewing ${\mathbb Z}$ as a subring of
${\mathfrak o}$ or ${\bar{\mathfrak o}}$, we have $p = \pi^e u$
for some unit $u$ in ${\bar{\mathfrak o}}$.

Elements $x\in {\bar{\mathfrak o}}$ have a unique power series
representation $x = \sum_{j\ge 0} x_j \pi^j$ with the $x_j$ lying
in a fixed set of representations of the elements of the residue class field
${\bar{\mathfrak o}}/\pi {\bar{\mathfrak o}}$. Like the
prime element $\pi$, the representations $\{x_j\}$ in
${\bar{\mathfrak o}}$ of the residue class field
can be chosen from the ring ${\mathfrak o}$
itself. For these elementary facts about
discrete valuation rings, see for example \cite{L} or \cite{L-2}.

The basic example is the ring of rational integers ${\mathfrak  o}
= {\Bbb Z}$ endowed with the $p$-adic valuation $|\cdot|_p$ for
some prime $p$. This is the setting of Theorems \ref{main-sum} and \ref{main-cong}.
More generally one can consider any Dedekind domain ${\mathfrak
o}$ with the finiteness property (FP) that the class fields
${\mathfrak o}/{\mathfrak p}$ are finite for all nonzero prime
ideals ${\mathfrak p}$. In this setting each nonzero prime ideal
${\mathfrak p}$ is maximal and gives rise to a discrete valuation
$|\cdot|_{\mathfrak p}$; in additive notation this valuation ${\rm
ord}_{\mathfrak p}$ is defined on ${\mathfrak o}$ so that
${\mathfrak p}^{{\rm ord}_{\mathfrak p}(x)}$ is the ${\mathfrak
p}$ factor in the prime ideal decomposition of the principal ideal
$x {\mathfrak o}$ generated by $x\in {\mathfrak o}$.
Furthermore the valuation ring
$\{x\in K: |x|_{\mathfrak p} \le 1 \}$ of the field
of fractions $K$ of ${\mathfrak o}$ is the localisation
of ${\mathfrak o}$ to ${\mathfrak p}$ and so we are in the
setting described in the previous paragraphs.

We denote by ${\bar{\mathfrak o}}_{\mathfrak p}$ the completion of
${\mathfrak o}$ with respect to the valuation arising from
${\mathfrak p}$ and we denote by $\pi_{\mathfrak p} \in {\mathfrak
o}$ the prime element generating the unique maximal ideal of
${\bar{\mathfrak o}}_{\mathfrak p}$. When the residue class field
${\mathfrak o}/{\mathfrak p}$ is finite, say with $q_{\mathfrak
p}$ elements, then via the isomorphism ${\mathfrak o}/{\mathfrak
p} \to {\bar{\mathfrak o}}_{\mathfrak p}/\pi_{\mathfrak p}
{\bar{\mathfrak o}}_{\mathfrak p}$, we see that the multiplicative
valuation $|x|_{\mathfrak p} := q_{\mathfrak p}^{-{\rm
ord}_{\mathfrak p}(x)}$, extended uniquely to ${\bar{\mathfrak
o}}_{\mathfrak p}$, is automatically normalised with
$|\pi_{\mathfrak p}|_{\mathfrak p} = q_{\mathfrak p}^{-1}$ or
${\rm ord}_{\mathfrak p}(\pi_{\mathfrak p}) = 1$.

In the setting of Dedekind domains with the finiteness property
(FP) many results from elementary number theory in ${\Bbb Z}$ have
analogues in this more abstract setting; see for example \cite{N}.
In a similar way there are analogous results of Theorems \ref{main-sum} and
\ref{main-cong}. Instead of $f\in {\Bbb Z}[X,Y]$, we consider
polynomials $f\in {\mathfrak o}[X,Y]$ where ${\mathfrak o}$ is any
Dedekind domain with the finiteness property (FP). As before, a
solution to the polynomial congruence $f \equiv 0$ mod ${\mathfrak
i}$ where ${\mathfrak i}$ is a nonzero ideal of ${\mathfrak o}$,
is defined to be an element in the class ring ${\mathfrak
o}/{\mathfrak i}$ and this ring is finite by the finiteness
property (FP). If we denote by $\|{\mathfrak i}\|$ the number of
elements of ${\mathfrak o}/{\mathfrak i}$, we study the normalised
number of solutions to the polynomial congruence $f\equiv 0$ mod
${\mathfrak i}$
$$
{\mathcal N}(f,{\mathfrak i}) \ := \ \|{\mathfrak i}\|^{-2} \# \{
f(x,y) \equiv 0 \, {\rm mod} \ {\mathfrak i} \} .
$$
If ${\mathfrak i} = \prod {\mathfrak p}^{{\rm ord}_{\mathfrak
p}({\mathfrak i})}$ is the prime ideal decomposition of the ideal
${\mathfrak i}$, then basic isomorphism theorems show
$$
{\mathcal N}(f,{\mathfrak i}) \ = \ \prod_{{\mathfrak p} |
{\mathfrak i}} {\mathcal N}(f, {\mathfrak p}^{{\rm ord}_{\mathfrak
p}({\mathfrak i})})
$$
where ${\mathcal N}(f, {\mathfrak p}^s) = q_{\mathfrak p}^{-2s} \#
\{ f\equiv 0 \ {\rm mod} \ {\mathfrak p}^s  \}$; see for example
\cite{L-2}. Therefore matters are reduced to the case when the
ideal ${\mathfrak i} = {\mathfrak p}^s$ is a power of a fixed
prime ideal ${\mathfrak p}$.

In this more abstract setting of Dedekind domains we also introduce and
study character sums which are generalisations of
the exponential sums ${\mathcal S}(f;p^s)$ over the integers ${\Bbb Z}$.
For a fixed nonzero prime ideal ${\mathfrak p}$, we consider
a nonprincipal additive character $\chi$ of the factor ring
${\mathfrak o}/{\mathfrak p}^s$ which we will assume to be
a {\it primitive} character in the sense that there exists
an element $y \in {\mathfrak o}$ with $|y|_{\mathfrak p} =
q_{\mathfrak p}^{-s+1}$
and so that $\chi(y + {\mathfrak p}^s) \not= 1$ (if no such element
exists, then $\chi$ would restrict to a nonprincipal character
of the factor ring ${\mathfrak o}/{\mathfrak p}^{s-1}$). For
an $f\in {\mathfrak o}[X,Y]$ (which by reducing the coefficients
mod ${\mathfrak p}^s$, we may view $f$ as a polynomial
with coefficients in ${\mathfrak o}/{\mathfrak p}^s$), we set
$$
{\mathcal S}_{\chi}(f;{\mathfrak p}^s) \ := \ q_{\mathfrak p}^{-2s} \, \sum\!\!\!\!\!\!\!\!\!\sum
\limits_{(x,y)\in [{\mathfrak o}/{\mathfrak p}^s]^2}
\chi(f(x,y))
$$
and, as in the setting of
${\mathfrak o} = {\Bbb Z}$, our main interest will be
to obtain bounds for ${\mathcal S}_{\chi}(f;{\mathfrak p}^s)$
and ${\mathcal N}(f;{\mathfrak p}^s)$ which are uniform over all
nonzero prime ideals ${\mathfrak p}$ and exponents $s$ when
$f\in {\mathfrak o}[X,Y]$ is a quasi-homogeneous polynomial; that
is, $f$ satisfies $f(r^{\kappa_1}x,r^{\kappa_2} y) = r f(x,y)$ for
some positive numbers $\kappa_1, \kappa_2 > 0$.

We now introduce the height $h(f)$ and Varchenko's exponents
$i(f)$ and $\nu(f)$ but appeal to the explicit description of
these parameters alluded to above, avoiding the original definitions
in terms of local coordinates. When  $f(x,y) = a x^{\alpha} y^{\beta}$
consists of a single monomial, we set as before
$h(f) = \max(\alpha,\beta)$. Furthermore we set
$i(f) = \nu(f) = 0$ when $\alpha \not= \beta$,
$\nu(f) = 1$ when $\alpha = \beta$ and
$i(f) = 1$ if $\alpha=\beta \ge 2$ but $i(f) = 0$
when $f(x,y) = a x y$. When
$f$ consists of more than one monomial, then $\kappa_1$ and $\kappa_2$
are uniquely determined by $f$ (see Lemma \ref{IM} below); of course
there is a continuum of choices for $\kappa_1$ and $\kappa_2$
when $f$ is a single monomial.
Recall that $K$ denotes the field of
fractions of ${\mathfrak o}$ and by
${\bar K}_{\mathfrak p}$, we denote
the field of fractions of ${\bar{\mathfrak o}}_{\mathfrak p}$. If
${\mathfrak o} = {\Bbb Z}$ and ${\mathfrak p} = p{\Bbb Z}$ for
some prime $p$, then $K = {\Bbb Q}$ and ${\bar K}_{\mathfrak p} =
{\Bbb Q}_p$.

Suppose now that $f$ consists of more than one monomial.
We will see that the zero set $\{f(x,y) = 0
\}$ of $f$ over some field extension of $K$
is a finite union of algebraic curves
or {\it roots of $f$} which can be enumerated by a certain
sequence of algebraic elements $\{\zeta_j\}$ over $K$, each {\it root}
$\zeta_j$ comes with an associated multiplicity or order $n_j$.
We define $m_K(f) := \max\{ n_j : \zeta_j \in K \}$, the maximal order
of the roots of $f$ over $K$, and following
\cite{IM}, we call $d(f) := (\kappa_1 + \kappa_2)^{-1}$
the {\it homogeneous distance} of $f$. Finally (as in \cite{IM})
we define the height of $f$ as
$$
h(f) \ := \ \max(m_K(f),d(f)).
$$
In Section \ref{prelim}, we will see that when ${\mathfrak o} =
{\Bbb Z}$, this definition of height coincides with the
euclidean definition in terms of the supremum of Newton distances.
In fact a result of Ikromov and M\"uller in \cite{IM} shows
that in the euclidean setting, the original definition
of the height $h(f)$ is equal $\max(m_{\Bbb R}(f),d(f))$ when
$f\in {\Bbb R}[X,Y]$ is any quasi-homogeneous polynomial with
real coefficients. Here $m_{\Bbb R}(f)$ is the maximal order of
the roots of $f$ over ${\Bbb R}$, instead of being over $K={\Bbb Q}$,
and therefore larger than $m_{\Bbb Q}(f)$ when $f\in {\Bbb Z}[X,Y]$.
Nevertheless taking the maximum with the homogeneous distance
$d(f)$ is the same; that is $h(f)$ is unchanged,
$h(f) =  \max(m_{\Bbb Q}(f),d(f))
= \max(m_{\Bbb R}(f),d(f))$. See Section \ref{prelim} for details.

Another result of Ikromov and M\"uller shows that the
Varchenko exponent $\nu(f)$ is equal to $0$ if $m_{\Bbb R}(f) \not= d(f)$
and $\nu(f) = 1$ if $m_{\Bbb R}(f) = d(f)$ when
$f\in{\Bbb R}[X,Y]$ is quasi-homogeneous (we recall that
the difference between $\nu(f)$ and $i(f)$ occurs only
when $h(f)<2$). We will see in Section \ref{prelim} that
in the setting of ${\mathfrak o} = {\Bbb Z}$, the dichotomy
$m_{\Bbb R}(f) = d(f)$ or $m_{\Bbb R}(f) \not= d(f)$ which
determines the exponents $i(f)$ and $\nu(f)$ is
exactly the same as $m_{\Bbb Q}(f) = d(f)$ or $m_{\Bbb Q}(f) \not= d(f)$ for
every quasi-homogeneous polynomial $f\in {\Bbb Z}[X,Y]$
except for the classes $E_m$. This explains
how the exceptional class $E_m$ arises and
indicates why the exponent of the linear factor $s$
in \eqref{basic-sum} depends on the prime $p$ for these
special polynomials.

In the abstract
setting of Dedekind domains ${\mathfrak o}$, we define
$\nu(f) = 0$ if $m_K(f) \not= d(f)$ and $\nu(f) = 1$ if
$m_K(f) = d(f)$. Furthermore we set $i(f) = \nu(f)$
except when $h(f) < 2$ where we always set $i(f) = 0$.
As in Theorems \ref{main-sum} and \ref{main-cong},
we obtain uniform estimates for
${\mathcal N}(f;{\mathfrak p}^s)$ and ${\mathcal S}_{\chi}(f;{\mathfrak p}^s)$
except when $f(x,y) = a (bx^2 + cxy + dy^2)^m$ for some
$m\ge 1$ and where the quadratic
polynomial $b x^2 + cxy + dy^2$  is irreducible over $K$
(the exceptional classes $E_m$ are restricted to $m\ge 2$
for the character sum ${\mathcal S}_{\chi}$). We will
continue to refer to these exceptional classes as $E_m$.

\begin{theorem}\label{main-abstract}
Let $f\in {\mathfrak o}[X,Y]$ be a quasi-homogeneous polynomial
with coefficients lying in a Dedekind domain ${\mathfrak o}$ with
property (FP). If the characteristic of ${\mathfrak
o}$ is positive, we assume that it is larger than the degree
of $f$. Then
there is a finite collection
${\mathcal P}(f)$ of prime ideals of ${\mathfrak o}$ and constants
$c', c, C > 0$, depending only on the degree of $f$, so that
for any
$f$ not in any exceptional class $E_m$ (and
$m\ge 2$ for ${\mathcal S}_{\chi}$),
\begin{equation}\label{main-sum-est-abstract}
|{\mathcal S}_{\chi}(f;{\mathfrak p}^s)| \ \le \ C s^{i(f)} q_{\mathfrak
p}^{-s/h(f)}
\end{equation}
and
\begin{equation}\label{main-poly-est-abstract}
c \, s^{\nu(f)} q_{\mathfrak p}^{-s/h(f)} q_{\mathfrak p}^{-2} \ \le
\ {\mathcal N}(f;{\mathfrak p}^s) \ \le \ C s^{\nu(f)} q_{\mathfrak
p}^{-s/h(f)}
\end{equation}
hold for every nonzero prime ideal ${\mathfrak p} \notin
{\mathcal P}(f)$ and $s\ge 1$. Furthermore for
${\mathfrak p} \notin {\mathcal P}(f)$,
\begin{equation}\label{abstract-poly-infinite}
c' \, s^{\nu(f)} q_{\mathfrak p}^{-s/h(f)} \ \le
\ {\mathcal N}(f;{\mathfrak p}^s)
\end{equation}
and, if also $f(x) \not= a x + b y$,
\begin{equation}\label{abstract-sum-infinite}
c' \, s^{i(f)} q_{\mathfrak p}^{-s/h(f)} \ \le
\ |{\mathcal S}_{\chi}(f;
{\mathfrak p}^s)|
\end{equation}
hold for infinitely many $s\ge 1$.

When $f$ belongs to some class $E_m$ (and $m\ge 2$
for ${\mathcal S}_{\chi}$), the
estimates \eqref{main-sum-est-abstract}, \eqref{main-poly-est-abstract},
\eqref{abstract-poly-infinite} and
\eqref{abstract-sum-infinite}
still hold but now the exponents
$i_{\mathfrak p}(f), \nu_{\mathfrak p}(f)$ depend on the
prime ideal ${\mathfrak p}$; precisely, if
$m\ge 2$, then $i_{\mathfrak p}(f) =
\nu_{\mathfrak p}(f) = 1$ or $0$ depending on whether the
two conjugate roots of $f$ lie in ${\bar K}_{\mathfrak p}$ or not,
respectively. If $m=1$, then $E_1$ is {\bf not} an exceptional
class for character sums (we have $i(f) = 0$ when $f\in E_1$)
but it is an exceptional class
for the problem of polynomial congruences; in this
case $\nu_{\mathfrak p}(f)$ depends on the prime ideal
${\mathfrak  p}$ and is defined as in the case $m\ge 2$.


\end{theorem}

As we have already mentioned, there are simple examples which show that the factor $q_{\mathfrak p}^{-2}$
in the lower bound in \eqref{main-poly-est-abstract} cannot
be replaced by $q_{\mathfrak p}^{-1}$. However if $m_K(f) \ge d(f)$,
then the factor $q_{\mathfrak p}^{-2}$ in
\eqref{main-poly-est-abstract} can be replaced by $q_{\mathfrak p}^{-1}$.

In the generality of Theorem \ref{main-abstract}, Cluckers
\cite{C-2} has proved the main estimate \eqref{main-sum-est-abstract} for quasi-homogeneous
polynomials in any number of variables which are non-degenerate
with respect to the Newton diagram (and in \cite{C-3}
for general quasi-homogeneous
polynomials when $s=1$ or $s=2$). In fact we will appeal
to Cluckers' result for certain cases when $h(f) < 2$. Alternatively
one can use more precise finite field character sum estimates
at the appropriate places in the arguments below.

{\it Acknowledgement}: We wish to thank Tony Carbery for some
motivating discussions at the beginning of these investigations.
Also we would like to thank Ben Lichtin for comments leading to a
more precise formulation of the main results.

\section{Notation and preliminaries}\label{prelim}
For any polynomial $g\in {\mathfrak o}[X,Y]$, $g(x,y) = \sum_{j,k}
c_{j,k} x^j y^k$, we call the set ${S}(g) := \{(j,k)\in
{\Bbb N}^2 : c_{j,k} \not= 0 \}$, the support of $g$. The {\it
Newton polyhedron} $\Delta(g)$ of $g$ is the convex hull of the
union of all quadrants $(j,k) + {\Bbb R}^2_{+}$ in ${\Bbb R}^2$
with $(j,k)\in {S}(g)$. If we use coordinates $(t_1,t_2)$
for points in the plane containing the Newton polyhedron, consider
the point $(d_{*},d_{*})$ in this plane where the bisectrix $t_1 = t_2$
intersects the boundary of $\Delta(g)$. The coordinate
$d_{*}$ is called the Newton distance of $g$ in the
coordinates $z = (x,y)$.

We turn our attention to quasi-homogeneous polynomials
$f\in {\mathfrak o}[X,Y]$ so that $f(r^{\kappa_1} x, r^{\kappa_2}
y) = r f(x,y)$ for some positive $\kappa_1, \kappa_2 > 0$ and all
$r\ge 0$. When $f(x,y) = a x^{\alpha}y^{\beta}$ is a single monomial,
the conclusions
of Theorem \ref{main-abstract} are easily verified in this case.
For the convenience of the reader we give the simple analysis in an appendix, see Section \ref{monomial}.
Therefore from now on (until the last section),
we assume that $f$ consists of more than one monomial. In
this case, it turns out that the dilation parameters
$\kappa_1$ and $\kappa_2$ are uniquely determined by $f$.

Recall that
$d(f) = (\kappa_1 + \kappa_2)^{-1}$ is the homogeneous
distance of $f$ and without loss of generality we will assume
$\kappa_2 \ge \kappa_1$. We record in the following lemma some
elementary facts about quasi-homogeneous polynomials observed in
\cite{IM}.

\begin{lemma}\label{IM} Let $f$ be a quasi-homogeneous
polynomial with dilation parameters $\kappa_1, \kappa_2 >0$
satisfying $\kappa_2 \ge \kappa_1$ and consisting
of more than one monomial. Then the exponents $\kappa_1 = t/m$, $\kappa_2 = r/m$ are rational numbers,
uniquely determined by $f$ with $gcd(r,t) = gcd(r,t,m) = 1$ (the
condition $\kappa_2 \ge \kappa_1$ means $r\ge t \ge 1$).
Furthermore $f(x,y) = x^{\alpha} y^{\beta} Q(x^r, y^t)$ for some
homogeneous polynomial $Q\in {\mathfrak o}[X,Y]$,
$$
Q(w_1, w_2) \ = \ a w_2^n + c_{n-1} w_2^{n-1} w_1 + \cdots + c_1
w_2 w_1^{n-1} + b w_1^n
$$
with $a,b \not= 0$. Factoring $Q(1,w) = a \prod_{j=1}^M (w -
\zeta_j)^{n_j}$ with respect to its distinct roots $\{\zeta_j\}$
lying in some extension field $L$ of $K$, we may write
\begin{equation}\label{f}
f(x,y) \ = \ a x^{\alpha} y^{\beta} \prod_{j=1}^M (y^t - \zeta_j
x^r)^{n_j} .
\end{equation}
Setting $n := \sum_{j=1}^M n_j$, we have
$$
d (f) \ = \ \frac{1}{\kappa_1 + \kappa_2} \ = \ \frac{t \alpha
+ r \beta + t r n}{t+ r}.
$$
\end{lemma}


We set $m_K(f) := \max(\alpha, \beta, \max(n_j : \zeta_j \in K ) )$
and $h(f) := \max(m_K(f), d(f))$.
If $\beta>0$, we introduce the index $j=0$ and set $\zeta_0 = 0$
and $n_0 = \beta$. We call the collection $\{\zeta_j \}_{j=0}^M$
the {\it roots} of $f$.

We have the following relationship between the multiplicities
$n_j, 0\le j\le M$ and the homogeneous distance $d(f)$;
the analogous result when $K = {\Bbb R}$ was
observed in \cite{IM}, the extension to general fields
$K$ being straightforward.

\begin{lemma}\label{IM-2} Let $f$ be a quasi-homogeneous
polynomial with dilation parameters $\kappa_1, \kappa_2 >0$
satisfying $\kappa_2\ge \kappa_1$. We use the notation introduced above (our underlying assumption  that $f$
is not a monomial remains in force).
\begin{enumerate}

\item If there is a multiplicity $n_{j_{*}} > d(f)$ for some
$0 \le j_{*} \le M$, then all the other multiplicities
must be {\bf strictly} less than $d(f)$; that is $n_j < d(f)$
for all $0 \le j \not= j_{*} \le M$. In particular,
there is at most one multiplicity $n_j, \ 0 \le j \le M$ with
$n_j > d(f)$.

\item If $\kappa_2/\kappa_1 \notin {\Bbb N}$ (that is, $t\ge 2$),
then $n = \sum_{j=1}^M n_j < d(f)$.


\item If $\kappa_2/\kappa_1 \in {\Bbb N}$, then $n_j \le d(f)$
for every $1\le j \le M$ with $\zeta_j \notin K$ (these are the
roots of $f$ with degree at least two with respect to $K$).

So necessarily, if there is a multiplicity $n_j > d(f)$ (unique
by (1)), it must correspond to a root $\zeta_j \in K$.

\item Finally, if there is a multiplicity $n_j$ corresponding to a
root $\zeta_j \notin K$ such that $n_j = d(f)$, then $f$ must
lie in an exceptional class $E_m$ for some $m\ge 1$.

Therefore outwith the special class of polynomials $f$ in
$E_m$, all multiplicities $n_j$ corresponding to a
root $\zeta_j \notin K$, necessarily satisfy $n_j < d(f)$.

\end{enumerate}
\end{lemma}

\begin{proof}
Suppose $n_{j_1} > n_{j_2} \ge d(f)$ for two
multiplicities $n_{j_1}, n_{j_2}$ with $0 \le j_1
\not= j_2 \le M$. Then
\begin{equation}\label{j_1}
d(f) \ \ge \ \frac{r(\beta + t n)}{r+t} \ > \ \frac{r
d(f) + r t d(f)}{r+1} \ = \ d(f) \frac{r(1+t)}{r+t} \ \ge \ d(f)
\end{equation}
which is a contradiction. This proves (1).

If $t\ge 2$, then $r > t\ge 2$ so that $\frac{1}{r} + \frac{1}{t} < 1$ and hence
$$
d(f) \ \ge \ \frac{r t n}{r + t} > n
$$
which proves (2).


To prove (3) we must show that $n_j \le d(f)$ for every $1\le j
\le M$ with $\zeta_j \notin K$. This simply follows from the fact
that the conjugates of any $\zeta_{j_0}$ over $K$ lie among the
roots $\{\zeta_j\}_{j=1}^M$ and any conjugate $\zeta_j$ of
$\zeta_{j_0}$ must have the same multiplicity; that is, $n_j =
n_{j_0}$. Finally if the degree of $\zeta_{j_0}$ is at least two,
then there exists a conjugate of $\zeta_{j_0}$ distinct from
itself and so if $n_{j_0} > d(f)$, one arrives at a
contradiction as in \eqref{j_1}.

Finally to prove (4), we first observe by (2) that necessarily
$\kappa_2/\kappa_1 \in {\Bbb N}$ and so $t=1$. If $n_j = d(f)$
is a multiplicity corresponding to a root $\zeta_j \notin K$, then
the degree of $\zeta_j$ over $K$ is at least two and so there is a
conjugate $\zeta_{j^{\prime}}$ of $\zeta_j$ with $j \not=
j^{\prime}$ and $n_j = n_{j^{\prime}}$. Therefore
$$
\frac{2r d(f)}{r+1} \ = \ \frac{r(n_j + n_{j^{\prime}})}{r+1} \
\le \ d(f)
$$
which implies that $r=1$. Now running the same argument again with
$r=1$, we obtain
$$
d(f) \ = \ \frac{(n_j + n_{j^{\prime}})}{2} \ \le \ \frac{\alpha
+ \beta + (n_j + n_{j^{\prime}} + n')}{2}  \ = \ d(f)
$$
where $n' = n - n_j - n_{j^{\prime}}$. This gives a contradiction
if there is a strict inequality above and so we conclude that
necessarily $\alpha = \beta = n' = 0$ which implies that $f$ is of
the form $E_m$ for some $m\ge 1$.

\end{proof}

\subsection{The height $h(f)$}
For real quasi-homogeneous polynomials $f \in {\Bbb R}[X,Y]$,
it was shown in \cite{IM-2} that $m_{\Bbb R}(f) := \max(\alpha,
\beta, \max(n_j : \zeta_j \in {\Bbb R})) = \sup_z d_z$
where $\sup_z d_z$ is the supremum of the Newton distances
of $f$ over all smooth local coordinate systems $z$ and
was introduced in the previous section
as the definition of the {\it height} for real-valued functions $f$.
Now when ${\mathfrak o} = {\Bbb Z}$, $K = {\Bbb Q}$ and
$f \in {\Bbb Z}[X,Y]$ is quasi-homogeneous, we see from
Lemma \ref{IM-2} that $\max(m_{\Bbb Q}(f), d(f)) =
\max(m_{\Bbb R}(f), d(f))$. In fact if $m_{\Bbb R}(f) > d(f)$,
then by part (3) of Lemma \ref{IM-2}, the unique multiplicity
$n_j > d(f)$ must correspond to a root $\zeta_j \in {\Bbb Q}$.
Hence our definition of
height as $\max(m_{\Bbb Q}(f),d(f))$ agrees with usual
definition $\sup_z d_z$ when $f \in {\Bbb Z}[X,Y]$
is quasi-homogeneous.

\subsection{The Varchenko exponents $i(f)$ and $\nu(f)$}\label{cases}

In the introduction we defined the Varchenko exponent
$\nu(f)$ as $1$ when $m_K(f) = d(f)$ and zero otherwise.
Furthermore the exponent $i(f) = \nu(f)$ whenever $h(f)\ge 2$
and we set $i(f) = 0$ when $h(f)<2$.

As we mentioned in the introduction,
Ikromov and M\"uller \cite{IM-2} showed
that the above definition of $\nu(f)$ in the setting
$K = {\Bbb R}$ is equivalent to the original definition
in terms of when the principal face of $f$ in adapted coordinates
is a vertex of the Newton diagram. Now when
${\mathfrak o} = {\Bbb Z}$ and $f \in {\Bbb Z}[X,Y]$
is quasi-homogeneous, we have already noted that the
implication $m_{\Bbb R}(f) \not= d(f) \Rightarrow
m_{\Bbb Q}(f) \not= d(f)$ follows from Lemma \ref{IM-2}
part (c). Furthermore Lemma \ref{IM-2} part (d)
shows that the implication $m_{\Bbb R}(f) = d(f)
\Rightarrow m_{\Bbb Q}(f) = d(f)$ holds
when $f$ does not belong to the exceptional class $E_m$.
Therefore outwith these exceptional classes $E_m$,
the dichotomy $m_{\Bbb R}(f) = d(f)$ or $m_{\Bbb R}(f) \not=
d(f)$ is the same as $m_{\Bbb Q}(f) = d(f)$ or
$m_{\Bbb Q}(f) \not= d(f)$ and so in the setting of
${\mathfrak o} = {\Bbb Z}$, the definition of the
Varchenko exponents agrees with the usual definition.










\subsection{The exceptional set ${\mathcal P}(f)$ of primes
 ideals in Theorem \ref{main-abstract}}\label{exceptional}

For a fixed nonzero prime ideal ${\mathfrak p}$
of  ${\mathfrak o}$, we will analyse
${\mathcal S}_{\chi}(f; {\mathfrak p}^s)$ and
${\mathcal N}(f;{\mathfrak p}^s)$ via the ${\mathfrak p}$-adic
valuation $|\cdot|_{\mathfrak p}$ on the field $K$, defined in the
introduction. In our analysis a finite collection
${\mathcal A}$
of algebraic elements over $K$ depending only on
$f$
arise naturally in our estimates.
At present we will not attempt to write the complete list
${\mathcal A}$ except to note that it includes
the roots $\{\zeta_j\}_{j\ge1}$ of $f$, differences of the roots
$\{\zeta_j - \zeta_k\}_{j\not= k}$ and the leading coefficient
$a$ in $Q$ defining $f$, introduced above. From time to time we will
add to it but in the end ${\mathcal A}$ will consist of finitely
many algebraic
elements over $K$, depending only on $f$.
The algebraic elements
in ${\mathcal A}$ all live in some algebraic closure
$K^{alg}$ of $K$ and for each valuation
$|\cdot|_{\mathfrak p}$ on $K$ there are many ways
to extend it to a valuation on $K^{alg}$. For each
prime ideal ${\mathfrak p}$ and element $\xi \in K^{alg}$,
we make the following canonical choice for $|\xi|_{\mathfrak p}$:
embed $K^{alg}$ into an algebraic
closure ${\bar K}_{\mathfrak p}^{alg}$
of ${\bar K}_{\mathfrak p}$ via an isomorphism over $K$.
There is a unique way of extending the valuation
$|\cdot|_{\mathfrak p}$ on $K$ to ${\bar K}_{\mathfrak p}^{alg}$
via ${\bar K}_{\mathfrak p}$ and for $\xi \in K^{alg}$,
we set $|\xi|_{\mathfrak p}$
to be the value of this extended valuation on the
image $\xi'$ of $\xi$ under this embedding. For notational convenience,
we will identify $K^{alg}$ with its embedded
image over $K$ in ${\bar K}_{\mathfrak p}^{alg}$ from now on.

More precisely it is the ${\mathfrak p}$-adic valuations
$|\cdot|_{\mathfrak p}$ of the elements in ${\mathcal A}$ which appear
in our estimates.
The important observation here is that there are only finitely
many prime ideals ${\mathcal P}(f)$ of ${\mathfrak o}$, depending
only on these algebraic elements over $K$ (and hence depends
only on $f$), so that
$|\xi|_{\mathfrak p} = 1$ for all ${\mathfrak p}\notin
{\mathcal P}(f)$ and every $\xi \in {\mathcal A}$.
To see this consider an algebraic element
$\xi \in {\mathcal A}$
and its minimal polynomial $x^d + a_{d-1} x^{d-1} + \cdots + a_0$
over $K$ so that
each $a_j \in K$. Let ${\mathcal F}(\xi)$ denote all the
prime ideals which arise in the prime ideal factorisation of one
of the (fractional) principal ideals $a_j {\mathfrak o}$. If
${\mathfrak p}\notin {\mathcal F}(\xi)$, then $|a_j|_{\mathfrak p}
= 1$ for every $0\le j \le d-1$ with  $a_j \not= 0$.
Fix a ${\mathfrak p} \notin
{\mathcal F}(\xi)$ and consider the conjugates $\xi_1, \xi_2, \ldots, \xi_r$ of
$\xi$ over $K$ lying in ${K}^{alg}$.
In the characteristic 0 case, $r=d$ and for each $s\ge 1$,
$$
a_{d-s} \ = \ \pm \sum\limits_{1\le j_1<\cdots<j_s\le d}
\xi_{j_1} \xi_{j_2} \cdots \xi_{j_s}.
$$
In the positive characteristic case, $d = r p^{\mu}$
for some $\mu\ge 0$ where $p$ is the characteristic of $K$.
Then for each $s\ge 1$,
\begin{equation}\label{symmetric}
a_{d-sp^{\mu}} \ = \ \pm \Bigl[\sum\limits_{1\le j_1<\cdots<j_s\le r}
\xi_{j_1} \xi_{j_2} \cdots \xi_{j_s} \Bigr]^{p^\mu}.
\end{equation}
We combine the two cases below and use \eqref{symmetric}
in both cases, taking $\mu = 0$ in \eqref{symmetric}
for the characteristic 0 case.

We claim that $|\xi|_{\mathfrak p} = 1$. In fact if
$|\xi_1|_{\mathfrak p} \ge |\xi_2|_{\mathfrak p} \ge
\cdots \ge |\xi_r|_{\mathfrak p}$, then we must have
equality $|\xi_1|_{\mathfrak p} = \cdots
= |\xi_r|_{\mathfrak p}$ and hence $|\xi|_{\mathfrak p} = 1$ since $a_0 = \pm [\xi_1 \cdots
\xi_r]^{p^{\mu}}$. Suppose equality
does not hold. Then $|\xi_1|_{\mathfrak p} = \cdots =
|\xi_s|_{\mathfrak p} > |\xi_{s+1}|_{\mathfrak p} \ge \cdots \ge |\xi_r|_{\mathfrak p}$ for some $1\le s <
r$ and hence
$1 = |a_{d-sp^{\mu}}| = |\xi_1 \cdots \xi_s|_{\mathfrak p}^{p^{\mu}}$
by the nonarchimedean nature of $|\cdot|_{\mathfrak p}$.
In fact
$|\xi_{j_1} \cdots \xi_{j_s}|_{\mathfrak p} < |\xi_1 \cdots \xi_s|_{\mathfrak p}$
for every term $\xi_{j_1} \cdots \xi_{j_s}$ in the
sum \eqref{symmetric} not equal to $\xi_1 \cdots \xi_s$.
This implies $1 = |\xi_1|_{\mathfrak p} = \cdots
= |\xi_s|_{\mathfrak p}$ which leads to the contradiction $1 = |a_0| = |\xi_1
\cdots \xi_r|_{\mathfrak p}^{p^{\mu}} < 1$.

We will also need to guarantee that $q_{\mathfrak p}$,
the number of elements of the residue class field
${\mathfrak o}/{\mathfrak p}$, is not too small; more precisely,
we will need that $q_{\mathfrak p} \ge C_{f}$ where $C_f$ is
a fixed positive constant, depending only on the degree
of our quasi-homogeneous polynomial $f$. The precise
value of $C_f$ will be determined later. In the setting
of Dedekind domains ${\mathfrak o}$ with the finiteness
property (FP) the collection
${\mathcal B}_C$
of prime ideals ${\mathfrak p}$
with absolute norm $q_{\mathfrak p} \le C$ is finite
in number; see for example, \cite{N}.

The exceptional set ${\mathcal P}(f)$ of prime ideals in the
statement of Theorem \ref{main-abstract} is the union of
${\mathcal F}(\xi)$ over all algebraic elements
$\xi\in {\mathcal A}$, together with the collection
${\mathcal B}_{C_f}$.

\subsection{Passing to the completion ${\bar{\mathfrak o}}_{\mathfrak p}$}\label{completion}
It will be convenient for us to pass to the completion
${\bar{\mathfrak o}}_{\mathfrak p}$. This will enable us to write
our character sum as an ``oscillatory integral'' over a local
field and to write the number of solutions to a polynomial
congruence as the measure of a sublevel set.
Since the residue class field ${\bar{\mathfrak o}}_{\mathfrak
p}/\pi_{\mathfrak p} {\bar{\mathfrak o}}_{\mathfrak p}$ is finite,
the ring ${\bar{\mathfrak o}}_{\mathfrak p}$ is then the compact
ring of integers of the local field ${\bar K}_{\mathfrak p}$, the
quotient field of ${\bar{\mathfrak o}}_{\mathfrak p}$. We then
have at our disposal a Haar measure $d\mu_{\mathfrak p}$ on ${\bar
K}_{\mathfrak p}$ which we normalise so that $\mu_{\mathfrak p}
({\bar{\mathfrak o}}_{\mathfrak p}) = 1$. The discrete valuation
$|\cdot|_{\mathfrak p}$, initially defined on ${\mathfrak o}$,
extends uniquely to a valuation on ${\bar K}_{\mathfrak p}$ which
we continue to denote by $|\cdot|_{\mathfrak p}$.

Since the prime ideal ${\mathfrak p}$ is fixed (although we keep
in mind estimates which are uniform in ${\mathfrak p}$), we will
suppress from now on the subscript ${\mathfrak p}$ in the various
quantities ${\bar{\mathfrak o}}_{\mathfrak p}, {\bar K}_{\mathfrak
p}, \pi_{\mathfrak p}, q_{\mathfrak p}, d\mu_{\mathfrak p},
|\cdot|_{\mathfrak p}$, etc... for notational convenience.

From the isomorphisms ${\mathfrak o}/{\mathfrak p}^s \to
{\bar{\mathfrak o}}/\pi^s {\bar{\mathfrak o}}$, we see that the
number of solutions to $f \equiv 0$ mod ${\mathfrak p}^s$ is the
same as the number of solutions to $f \equiv 0$ mod $\pi^s
{\bar{\mathfrak o}}$; that is
$$
{\mathcal N}(f; {\mathfrak p}^s) \ = \ q^{-2s} \# \{ f \equiv 0 \,
{\rm mod} \, \pi^s {\bar{\mathfrak o}} \} \ = \ {\mathcal N}(f;
\pi^s {\bar{\mathfrak o}}).
$$
If $d\mu_2 = d\mu \times d\mu$ denotes the product measure on
${\bar{\mathfrak o}} \times {\bar{\mathfrak o}}$, we have
\begin{equation}\label{sublevel-congruence}
{\mathcal N}(f; \pi^s{\bar{\mathfrak o}}) \ = \ \mu_2 \bigl(\{z \in {\bar{\mathfrak
o}}\times {\bar{\mathfrak o}} : \ |f(z)| \ \le \ q^{-s} \}\bigr).
\end{equation}
In fact the right hand side of \eqref{sublevel-congruence} is
equal to
$$
\int\!\!\!\!\int_{{\bar{\mathfrak o}}\times {\bar{\mathfrak o}}} {\bf
1}_{\{|f(w)|\le q^{-s}\}}(y) \, d\mu_2(y) \ = \ \sum_{z'\le \pi^s
{\bar{\mathfrak o}}} \int\!\!\!\!\int_{B_{q^{-s}}(z')} {\bf 1}_{\{|f(w)|\le
q^{-s}\}}(y) \, d\mu_2(y)
$$
$$
= \ q^{-2s} \# \{ z' \le \pi^s {\bar{\mathfrak o}} : \ |f(x',y')|
\ \le \ q^{-s} \} \ = \ {\mathcal N}(f; \pi^s{\bar{\mathfrak o}}).
$$
Here we are using the nonstandard notation $z' \le \pi^s
{\bar{\mathfrak o}}$ to denote elements $z' = (x',y')$ in
${\bar{\mathfrak o}} \times {\bar{\mathfrak o}}$ of the form $x' =
x_0 + x_1 \pi + \cdots + x_{s-1} \pi^{s-1}, y' = y_0 + y_1 \pi +
\cdots + y_{s-1} \pi^{s-1}$ where each $x_j$ and $y_j$ varies over
the $q$ representations in ${\mathfrak o}$ of the elements in the
residue class field. Also $B_r(z) = \{w\in {\bar K}\times {\bar K}
: \|w-z\| \le r\}$ denote balls in ${\bar K}\times {\bar K}$ where
$\|z\| := \max(|x|,|y|)$ if $z = (x,y)$. The second equality above
follows since $|f(x,y)| \le q^{-s}$ if and only if $|f(x',y')| \le
q^{-s}$ for elements $z = (x,y) \in B_{q^{-s}}(x',y')$.

A similar identity holds for character sums. We claim
we can find a non-principal additive character
$\psi$ on ${\bar K}$ with
$\psi \equiv 1$ on ${\bar{\mathfrak o}}$ so that
\begin{equation}\label{osc-S-chi}
{\mathcal S}_{\chi}(f; {\mathfrak p}^s) \ = \ \int\!\!\!\int_{{\bar{\mathfrak o}}\times {\bar{\mathfrak o}}} \, \psi(\pi^{-s} f(x,y)) \, d\mu(x)
d\mu(y).
\end{equation}
Furthermore $\psi$
will be non-trivial on $\{z \in {\bar K} : |z|\le q\}$ since
$\chi$ is primitive.

In fact, starting with our non-principal, primitive additive character $\chi$ on ${\mathfrak o}/{\mathfrak p}^s$, we pass to a character $\chi'$
on ${\bar{\mathfrak o}}/\pi^s {\bar{\mathfrak o}}$ via the isomorphism
${\mathfrak o}/{\mathfrak p}^s \simeq {\bar{\mathfrak o}}/\pi^s
{\bar{\mathfrak o}}$ so that
$$
{\mathcal S}_{\chi}(f;{\mathfrak p}^s) \ = \ q^{-2s} \, \sum\!\!\!\!\!\!\!\!\!\sum
\limits_{(x,y)\in [{\mathfrak o}/{\mathfrak p}^s]^2}
\chi(f(x,y)) \ = \ q^{-2s}
\sum\!\!\!\!\!\!\!\!\!\sum
\limits_{(x,y)\in [{\bar{\mathfrak o}}/\pi^s{\bar{\mathfrak o}}]^2}
\chi'(f(x,y)).
$$
Next, $\chi'$ restricts
to a non-principal character ${\bar \chi}$ on ${\bar {\mathfrak o}}$
via ${\bar{\chi}}(x) = \chi'(x+\pi^s{\bar{\mathfrak o}})$
which is equal to 1 on $\pi^{s} {\bar {\mathfrak o}}$. The characters of ${\bar{\mathfrak o}}$
arise as $x\to\psi_0(yx)$ for some $y = \sum_{j=-m}^{-1} x_j \pi^j$ and some fixed non-principal character
$\psi_0$ on ${\bar K}$
which is 1 on ${\bar{\mathfrak o}}$ and non-trivial on
$\{z\in {\bar K}: |z|\le q\}$. Hence ${\bar{\chi}}(x) = \psi_0(y' x)$ for some $y'$ satisfying
$|y'| = q^{s}$. In fact since $\psi_0$ is non-trivial on $B_q(0)$ we can
find an $x$ with $|y' x| =q$ so that ${\bar \chi}(x) \not=1$ and hence $|x| \ge q^{-s +1}$
implying $|y'|\le q^{s}$. On the other hand since
$\chi$ is a primitive character, we can find a $v$ with
$|v| = q^{-s +1}$ so that $\psi_0(y' v) = {\bar \chi}(v) \not= 1$.
This implies that $|y'| q^{-s+1} = |y'v| \ge q$ and so $|y'| \ge q^{s}$.

Therefore the character $\psi(z) := \psi_0 (y' \pi^{s} z)$
on ${\bar K}$ has the properties $\psi(\pi^{-s} x) = {\bar {\chi}}(x)$ on ${\bar{\mathfrak o}}$,
$\psi \equiv 1$ on ${\bar{\mathfrak o}}$ and $\psi$ is non-trivial on $B_q(0)$. Furthermore, using the nonstandard notation
$z' \le \pi^s {\bar{\mathfrak o}}$ introduced above,
$$
\int\!\!\!\!\int_{{\bar{\mathfrak o}}\times{\bar{\mathfrak o}}} \psi(\pi^{-s} f(z)) d\mu_2(z) =
\sum_{w' \le \pi^{s}{\bar{\mathfrak o}}}
\int\!\!\!\!\int_{B_{q^{-s}} (w')}
\psi(\pi^{-s} f(z)) d\mu_2(z) \ = \
$$
$$
q^{-2s}\sum_{w'\le \pi^{s}{\bar{\mathfrak o}}}
\psi(\pi^{-s} f(w'))
= q^{-2s} \sum_{x'\le \pi^{s}{\bar{\mathfrak o}}}
{\bar {\chi}}(f(x')) =
q^{-2s}
\sum\!\!\!\!\!\!\!\!\!\!\!\sum
\limits_{(x,y)\in [{\bar{\mathfrak o}}/\pi^s{\bar{\mathfrak o}}]^2}
\chi'(f(x,y))
$$
which establishes \eqref{osc-S-chi} since the last
sum is equal to ${\mathcal S}_{\chi}(f;{\mathfrak p}^s)$.

\subsection{Lower bounds on the
distance from roots of $f$
to ${\bar{\mathfrak o}}$}

In our analysis we will need to understand sets of the form
$B_{\rho}(\zeta) \cap {\bar{\mathfrak o}}$ where $\zeta$
is one of the nonzero roots of
$f$ appearing in \eqref{f},
$$
B_{\rho}(\zeta) \ = \ \{y\in {\bar K}^{alg} : |y-\zeta|\le \rho \}
$$
is a ball lying
in ${\bar K}^{alg}$ and $|\cdot|$
is the unique extension to ${\bar K}^{alg}$ of our original valuation
on ${\bar K}$. For ${\mathfrak p} \notin {\mathcal P}(f)$,
$\zeta$ has the property
that $|\zeta|=|\zeta|_{\mathfrak p} = 1$ as well as $|\zeta_1| = \cdots = |\zeta_r| = 1$
where $\zeta_1, \ldots, \zeta_r$ denote the conjugates of $\zeta$
over $K$. Since these conjugates are among the roots of $f$, they
lie
in ${\mathcal A}$ as well as their differences and so
$|\zeta_s - \zeta_t| = |\zeta_s - \zeta_t|_{\mathfrak p} = 1$ also holds for $1\le  s\not= t \le r$
and ${\mathfrak p}\notin {\mathcal P}(f)$.

It will be useful to have a good bound
from below on the quantity $\inf_{x\in{\bar{\mathfrak o}}} |x-\zeta|$
whenever $\zeta \notin {\bar K}$. This will be easily
achieved by Krasner's lemma
when $\zeta$ is separable over $K$.
In this case we will see that $\inf_{x\in {\bar{\mathfrak o}}} |x-\zeta| = 1$. In fact we will show that $|x-\zeta| = 1$
for every $x \in {\bar {\mathfrak o}}$.
When $\zeta$ is not separable over $K$ (and so $K$ must
have positive characteristic, say equal to $p$), then there
is a $\mu\ge 1$ such that $\zeta^{p^{\mu}}$ is separable over
$K$. We claim that $\zeta^{p^{\mu}} \notin {\bar K}$ and so
one can argue again by Krasner's lemma to deduce that
$\inf_{x\in{\bar{\mathfrak o}}} |x-\zeta| = 1$. In fact
the supposition $\zeta^{p^{\mu}} \in {\bar K}$ implies that
$\zeta \in {\bar K}$, contrary to our
assumption $\zeta \notin {\bar K}$. To see this, note that
in the positive characteristic case, ${\bar K} = {\Bbb F}_q((\pi))$
is the field of Laurent series with coefficients in the
finite field ${\Bbb F}_q$ with $q = p^f$ elements. Since
$\zeta$ is algebraic over $K$ it is algebraic over ${\bar K}$
and so lies in ${\Bbb F}_{q^d}((\pi))$ for some $d\ge 2$
since $\zeta \notin {\Bbb F}_q((\pi))$. Hence
$$
\zeta \ = \ \xi_0 + \xi_1 \pi + \xi_2 \pi^2 + \xi_3 \pi^3 + \cdots
$$
where each $\xi_j$ lies in ${\Bbb F}_{q^d}$  (recall that
$|\zeta| = |\zeta|_{\mathfrak p} = 1$) and so
$$
\zeta^{p^{\mu}} \ = \ [\xi_0]^{p^{\mu}} + [\xi_1]^{p^{\mu}}
\pi^{p^{\mu}} + [\xi_2]^{p^{\mu}} \pi^{2p^{\mu}} + \cdots \ \in
\ {\bar K} \ = \ {\Bbb F}_q((\pi)).
$$
Therefore $\xi_j^{p^{\mu}} \in {\Bbb F}_q$ for each $j\ge 0$.
The map $\phi(x) = x^{p^{\mu}}$ is automorphism for both
fields ${\Bbb F}_{q^d}$ and ${\Bbb F}_q$. As an automorphism
of ${\Bbb F}_q$, we can find an $\eta_j \in {\Bbb F}_q$
such that $\eta_j^{p^{\mu}} = \xi_j^{p^{\mu}}$ for each $j\ge 0$.
As an automorphism of ${\Bbb F}_{q^d}$, we deduce
$\xi_j = \eta_j \in {\Bbb F}_q$ for every $j\ge 0$, implying
that $\zeta \in {\Bbb F}_q((\pi)) = {\bar K}$, contradicting
our underlying assumption $\zeta \notin {\bar K}$.

We have the following lemma.

\begin{lemma}\label{krasner} In the setting above, suppose
$\zeta \notin {\bar K}$. Then for every
$x\in{\bar{\mathfrak o}}, \  |x-\zeta|  =  1$.
\end{lemma}
\begin{proof}
We split the proof into two cases. First suppose that
$\zeta$ is separable over $K$. Then $\zeta$ is separable
over ${\bar K}$. Suppose that there is an $x \in {\bar{\mathfrak o}}$
such that $|x - \zeta| < 1$. Then $|\zeta-x| < |\zeta - \zeta'|$
for all conjugates $\zeta'$ of $\zeta$ over ${\bar K}$ different
from $\zeta$ since $|\zeta - \zeta'| = 1$. Hence by
Krasner' lemma (see for example \cite{L-2}), ${\bar K}[\zeta]
\subset {\bar K}[x] = {\bar K}$ implying $\zeta \in {\bar K}$
which contradicts our underlying assumption $\zeta \notin {\bar K}$.
Hence $|x - \zeta| \ge 1$ but clearly
$|x - \zeta| \le 1$ since
$|\zeta| = 1$.

Now let us consider the case when $\zeta$ is not separable
over $K$. As discussed above, there is a $\mu \ge 1$ so that
$\zeta^{p^{\mu}}$ is separable over $K$ yet does not belong
to ${\bar K}$. Hence we can apply Krasner's lemma to $\zeta^{p^{\mu}}$
to deduce $|y - \zeta^{p^{\mu}}| \ge 1$ for every $y \in
{\bar{\mathfrak o}}$. Therefore for every $x\in {\bar{\mathfrak o}}$,
$$
|x - \zeta|^{p^{\mu}} \ = \ |(x-\zeta)^{p^{\mu}}| \ = \
|x^{p^{\mu}} - \zeta^{p^{\mu}}| \ \ge \ 1
$$
and so we have $|x - \zeta| = 1$.

\end{proof}

\subsection{Notation and constants}

All constants $C,c,c'>0$ throughout this paper
will depend only on the degree of our quasi-homogeneous
polynomial $f(x,y)$,
although the values of these constants may change from line to
line. Often it will be convenient to suppress explicitly
mentioning the constants $C$ or $c$ in these inequalities and we
will use the notation $A \lesssim B$ between positive quantities
$A$ and $B$ to denote the inequality $A \le C B$ or $c A \le B$.
Finally we use the notation $A \sim B$ to denote that both
inequalities $A \lesssim B$ and $B \lesssim A$ hold.

\section{Polynomial congruences in one variable}

The proof of Theorem \ref{main-abstract} relies on a precise structural statement for
sublevel sets of polynomials $P\in {\bar{K}}[X]$ in one variable with
coefficients lying in our local field ${\bar{K}}$ which
carries the nontrivial valuation $|\cdot|$. Suppose our
polynomial $P(x) = a \prod (x - \xi_j)^{e_j}$ has
distinct roots $\xi_1, \ldots, \xi_m$ lying in
${\bar K}^{alg}$. As remarked earlier our valuation on
${\bar K}$ extends uniquely to a valuation on ${\bar K}^{alg}$
which we will continue to denote by $|\cdot|$. The
structural statement of a sublevel set $\{x\in S: \, |P(x)|\le
\delta \}$ where $S\subset {\bar{K}}$
will be given in terms of {\it balls} $B_r(\xi) := \{y\in
{\bar K}^{alg}: |y-\xi| \le r \}$ in ${\bar K}^{alg}$, centred at the roots $\{\xi_j\}$ of
$P$ with radii $r$ described by {\it root clusters} ${\mathcal C}$
of $\{\xi_1, \ldots, \xi_m\}$; a root cluster ${\mathcal C}$ being
defined simply as some subset ${\mathcal C}\subset \{\xi_1,
\ldots, \xi_m\}$ of the roots of $P$. We associate a size
$S({\mathcal C}) := \sum_{\xi_j \in {\mathcal C}} e_j$ to a
cluster by counting the number of roots in the cluster with
multiplicities. The following proposition is the non-archimedean
version of a basic sublevel set estimate due to Phong, Stein and
Sturm \cite{PSS} and its simple proof can be found in \cite{W-1}
or \cite{W-2}.

\begin{proposition}\label{PS}
With the notation as above, we have
\begin{equation}\label{cluster}
\bigl\{x\in S : \ |P(x)|\le \delta \bigr\} \ = \
\bigcup_{j=1}^m \, [B_{r_j}(\xi_j) \cap S ] .
\end{equation}
Here
$$
r_j \ = \ \min_{{\mathcal C}\ni \xi_j} r_{{\mathcal C},j}(\delta)
\ := \ \min_{{\mathcal C}\ni \xi_j} \ \Bigl[\frac{\delta}{|a
\prod_{\xi_k \notin {\mathcal C}} (\xi_j - \xi_k)^{e_k}
|}\Bigr]^{1/S({\mathcal C})}
$$
where the minimum is taken over all root clusters ${\mathcal C}$
containing $\xi_j$ and the product is taken over all $k$ such that
$\xi_k \notin {\mathcal C}$.
\end{proposition}

In our application of Proposition \ref{PS} the roots $\xi_j$
and the coefficient $a\in {\bar K}$ have the property that
$|a| = |\xi_j| = 1$. Furthermore $|\xi_j - \xi_k| = 1$ for all $j\not= k$. In this case
the minimum over all root clusters ${\mathcal C}$ containing a root
$\xi_j$ in the definition of $r_j$ is attained when ${\mathcal C} = \{\xi_j\}$
is the singleton root cluster (here we are assuming $\delta \le 1$). Therefore
\begin{equation}\label{rj}
r_j \ = \ r_j(\delta) \ = \ \delta^{1/e_j}.
\end{equation}

\section{The main estimates in Theorem \ref{main-abstract}}\label{homogeneous}

We will give a direct proof which treats simultaneously
our character sums
$$
S_{\chi}(f; {\mathfrak p}^s) \ = \ \int\!\!\!\int_{{\bar{\mathfrak o}}\times
{\bar{\mathfrak o}}} \psi(\pi^{-s} f(x,y)) \, d\mu(x) d\mu(y)
$$
and the number of solutions to our polynomial congruences
$$
{\mathcal N}(f; {\mathfrak p}^s) = \mu_2 \bigl(\{z \in {\bar{\mathfrak
o}}\times {\bar{\mathfrak o}} : \ |f(z)| \ \le \ q^{-s} \}\bigr) =
\int\!\!\!\int_{{\bar{\mathfrak o}}\times
{\bar{\mathfrak o}}} {\bf 1}_{\bar{\mathfrak o}}(\pi^{-s} f(x,y)) \, d\mu(x) d\mu(y);
$$
that is one proof works to establish the estimates
\eqref{main-sum-est-abstract}, \eqref{main-poly-est-abstract},
\eqref{abstract-poly-infinite}
and \eqref{abstract-sum-infinite} in
Theorem \ref{main-abstract}.
We make the ongoing assumption that $f$ consists of
more than one monomial; for the simple case when
$f$ is a monomial, see Section \ref{monomial}.
When treating ${\mathcal S}_\chi(f;{\mathfrak p}^s)$, we
assume $2 \le d(f)$ (for the case $d(f)<2$,
see Section \ref{two-cases}).
Recall that the Varchenko exponents
$i(f) = \nu(f)$ agree when $h(f)\ge 2$ and in this case,
we will denote the common
value as $\nu(f)$.

In the case when the ratio $\kappa_2/\kappa_1$ of the dilation
parameters is an integer (this includes the case when
$f$ is a a homogeneous polynomial),
the argument below works if
we replace the
additive character $\psi$ or the indicator function ${\bf 1}_{\bar{\mathfrak o}}$
by any general complex-valued
function ${\mathcal C} : {\bar K} \to {\Bbb C}$ with the following
properties:

$(C1) \ \ \ \ {\mathcal C} \ \equiv \ 1$ \ on ${\bar{\mathfrak
o}}$;

$(C2)$ \ \ ${\displaystyle \int_{|x|=1} {\mathcal C}(\pi^{-j}
g(x)) \ d\mu(x) \ = \ 0 \ \ {\rm for \ all} \ \ j\ge 2} \
\ {\rm and} \ \ g(x) = b_k x^k + \cdots + b_1 x + b_0$
\ \ \ \ \ \ \ \ in ${\mathfrak o}[X]$
with the property $\pi | \ b_r$ for $r\le k-1$ and $|k b_k| = 1$; and

$(C3)$ \ \ ${\displaystyle \bigl| \int_{|x|=1} {\mathcal
C}(\pi^{-1} b x^k) \ d\mu(x)\bigr| \ \le \ A_k q^{-1/2} \ \ {\rm
for \ all} \ k\ge 1 \ \ {\rm and} \ \ |b|=1}$.

When ${\mathcal C} = {\bf 1}_{\bar{\mathfrak o}}$,
properties (C2) and (C3) are trivial to verify since $\pi^{-j} g(x)
\notin {\bar{\mathfrak o}}$ for all $j\ge 1$ and $|g(x)|=1$
 whenever $|x|=1$; in particular the integral in (C3) in fact vanishes in this case.

If we have this extra cancellation we can replace (C3)
with the stronger property

$(C3)'$ \ \   ${\displaystyle \int_{|x|=1} {\mathcal C}(\pi^{-1} b
x^k) \ d\mu(x) \ = \ 0 \ \ {\rm for \ all} \ \ k\ge 1 \
\ {\rm and} \ \ |b|=1}$.

For the general case, when $\kappa_2/\kappa_1$ is not
necessarily an integer, we will need a slight strengthening
of property (C2); namely the vanishing
$\int_{|x|=1} {\mathcal C}(\pi^{-j}
g(x)) d\mu(x) = 0$ still holds if the region of
integration $\{|x|=1\}$ is replaced by any finite
union of disjoint balls $B_{q^{-1}}$ in ${\bar{\mathfrak o}}$
with radius $q^{-1}$. Furthermore there will be one instance
where we will need to appeal to a stronger form of property
(C3); namely when the monomial $x^k$ is replaced by a
general monic $x^k + c_{k-1}x^{k-1} + \cdots + c_1 x$
polynomial. Such square root $q^{-1/2}$ estimates for
character sums follow
from the work of A. Weil but there is no need to appeal
to such deep results in the monomial case $x^k$.

The verification of (C2) and (C3) for the additive character
${\mathcal C} = \psi$ is straightforward. In fact to see (C2), we
write our oscillatory integral
$$
\int_{|x|=1} {\psi}(\pi^{-j} g(x)) \ d\mu(x) \ = \ q^{-j} \sum_{
\begin{array}{c}\scriptstyle
x\le \pi^j {\bar{\mathfrak o}}\\
         \vspace{-5pt}\scriptstyle \pi \not| \ x
         \end{array} }
\chi^{\prime}(\pi^{-j} g(x))
$$
back in terms of a character sum over ${\bar{\mathfrak o}}/\pi^j {\bar{\mathfrak o}}$
(using the nonstandard notation $x \le \pi^j {\bar{\mathfrak o}}$ introduced in Section
\ref{prelim}) and decompose the sum via $x = z + \pi^{j-1} y$ so that
$$
\int_{|x|=1} {\psi}(\pi^{-j} g(x)) \ d\mu(x) \ = \ q^{-j} \sum_{
\begin{array}{c}\scriptstyle
z\le \pi^{j-1} {\bar{\mathfrak o}}\\
         \vspace{-5pt}\scriptstyle \pi \not| \ z
         \end{array} }
\chi^{\prime}(\pi^{-j} g(z)) \sum_{y\le \pi {\bar{\mathfrak o}}}
\chi^{\prime}(\pi^{-1} g'(z) y);
$$
in fact for $x = z + \pi^{j-1} y$, we have $g(x) \equiv g(z) + g'(z) \pi^{j-1} y$ mod $\pi^j {\mathfrak o}$
since $j\ge 2$. Furthermore $g'(z) \equiv k b_k z^{k-1}$
mod $\pi {\bar{\mathfrak o}}$ and so for each $z$ arising in the first sum
above, the inner sum can written as
$$
\sum_{y\le \pi {\bar{\mathfrak o}}}
\chi^{\prime}(\pi^{-1} k b_k z^{k-1} y)
$$
and this sum vanishes
by the basic orthogonality
property of the nonprincipal character $\chi^{\prime}$ since
$\pi \not| \ k b_k z^{k-1}$ when $|k b_k | = 1$.
Note that the proof works if the region of integration
$\{|x|=1\}$ in the original integral is replaced by
any union of disjoint balls $B_{q^{-1}}$ with radius $q^{-1}$.

As mentioned above, property (C3) for character sums
follows from the work of A. Weil but in this case there
is a simple well-known proof which we reproduce for
the convenience of the reader. Writing again
the oscillatory integral
$$
\int_{|x|=1} {\psi}(\pi^{-1} b x^k) \ d\mu(x) \ = \ q^{-1} \sum_{
x \le \pi: \ x\not= 0}
\chi(\pi^{-1} b x^k)
$$
as a character sum over the
finite field ${{\mathfrak o}}/\pi {{\mathfrak o}}$, we rewrite
the right-hand side as
$$
q^{-1} \sum_{z \le \pi: \ z\not= 0}
\chi(\pi^{-1} b z) \sum_{x: \ x^k = z} 1.
$$
Let $g$ generate the multiplicative group
$[{\mathfrak o}/\pi {\mathfrak o}]\setminus \{0\}$ and write
$\log(z)$ as the integer $\ell$, $0\le \ell \le q-2$ such that
$z = g^{\ell}$. For a given $z$, the equation $x^k = z$
is solvable if and only if $gcd(k,q-1) \, | \ \log(z)$ and when
this happens there are precisely $d:= gcd(k,q-1)$ solutions.
Therefore we can
write the inner sum above as $\sum_{m=0}^{d-1} \exp(2\pi i m \log(z)/ d)$
and hence
$$
q^{-1} \sum_{x \le \pi: \ x\not= 0}
\chi(\pi^{-1} b x^k) \ = \ q^{-1} \sum_{m=0}^{d-1} \ \ \
\sum_{z\le \pi: \ z\not= 0} \exp(2\pi i m \log(z)/ d) \ \chi(\pi^{-1} b z).
$$
For $m=0$, the character sum in $z$ gives $-1$ and for $m\ge 1$,
the quantity
$$
U_m \ = \ \sum_{z\le \pi: \ z\not= 0} \exp(2\pi i m \log(z)/ d) \ \chi(\pi^{-1} b z)
$$
has modulus equal to $\sqrt{q}$. In fact
$$
|U_m|^2 \ = \ \sum_{x\not= 0} \sum_{y\not= 0}\exp(2\pi i m \log(x)/ d)
\chi(\pi^{-1} b y(x-1))
$$
and the double sum is equal to
$$
q - 1 + \sum_{x\not= 0,1} \exp(2\pi i m \log(x)/d)
\sum_{y\not=0} \chi(\pi^{-1} b (x-1) y) =
q - 1 - \sum_{x\not= 0,1} \exp(2\pi i m \log(x)/d) = q.
$$
This shows that property (C3)
$$
\Bigl|\int_{|x|=1} {\psi}(\pi^{-1} b x^k) \ d\mu(x) \Bigr| \ \le \
\sqrt{gcd(k,q-1)} \ q^{-1/2}
$$
holds when ${\mathcal C} = \psi$.

\subsection{The basic decomposition} Recall that
$$
f(x,y) \ = \ a x^{\alpha} y^{\beta} \prod_{j=1}^M (y^t - \zeta_j x^r)^{n_j}
$$
and $d(f) = (\alpha t + \beta r + n r)/(r+t)$ where $n = \sum_{j=1}^M n_j$.
Following Denef and Sperber \cite{DS} we decompose the integral
$$
\int\!\!\!\int_{{\bar{\mathfrak o}}\times
{\bar{\mathfrak o}}} {\mathcal C}(\pi^{-s} f(x,y)) \, d\mu(x) d\mu(y) \ = \
\sum_{k_1,k_2 \ge 0} \int\!\!\!\int_{|x|=q^{-k_1}, |y|=q^{-k_2}}
{\mathcal C}(\pi^{-s} f(x,y)) \, d\mu(x) d\mu(y)
$$
into three parts $I + II + III$ where
$$
I \ = \ \sum_{
\begin{array}{c}\scriptstyle
k_1, k_2 \ge 0\\
         \vspace{-5pt}\scriptstyle t k_2 = r k_1
         \end{array} }
\int\!\!\!\int_{|x|= q^{-k_1}, |y|=q^{-k_2}}
{\mathcal C}(\pi^{-s} f(x,y)) \, d\mu(x) d\mu(y)
$$
$$
\ = \ \sum_{m\ge 0} q^{-(t+r)m} \int\!\!\!\int_{|x|=1,|y|=1} {\mathcal C}(\pi^{-s+mN} f(x,y)) \,
d\mu(x) d\mu(y)
$$
and $N := \alpha t + \beta r + n t r$. Here we used the
fact that $t$ and $r$ are relatively prime so that
if $ t k_2 = r k_1$, then $t | \, k_1$ and so we
are writing $k_1 = m t$ in the above sum.
Furthermore $II$ and $III$ are
defined as
$$
II \ = \   \sum_{
\begin{array}{c}\scriptstyle
k_1, k_2 \ge 0\\
         \vspace{-5pt}\scriptstyle t k_2 < r k_1
         \end{array} }
\int\!\!\!\int_{|x|= q^{-k_1}, |y|=q^{-k_2}}
{\mathcal C}(\pi^{-s} f(x,y)) \, d\mu(x) d\mu(y),
$$
$$
III \ = \   \sum_{
\begin{array}{c}\scriptstyle
k_1, k_2 \ge 0\\
         \vspace{-5pt}\scriptstyle t k_2 > r k_1
         \end{array} }
\int\!\!\!\int_{|x|= q^{-k_1}, |y|=q^{-k_2}}
{\mathcal C}(\pi^{-s} f(x,y)) \, d\mu(x) d\mu(y).
$$

\subsection{Estimates for $I$}\label{I}
We turn our attention first to I which is the main term. We
split $I = I_1 + I_2$ into two parts where
$$
I_1 \ = \
\sum_{
m N \ge s }
q^{-(t+r)m} \int\!\!\!\!\!\int\limits_{|x|=1 \, |y|=1} {\mathcal C}(\pi^{-s+mN} f(x,y)) \,
d\mu(x) d\mu(y)
$$
and
$$
I_2 \ = \ \sum_{
          m N \le s - 1
          }
q^{-(t+r)m} \int\!\!\!\!\!\int\limits_{|x|=1 \, |y|=1} {\mathcal
C}(\pi^{-s + mN} f(x,y)) \, d\mu(y) d\mu(x).
$$
From property (C1) we see that
\begin{equation}\label{I1-above}
 I_1 \ = \ (1-q^{-1})^2 \sum_{
          m N \ge s }
q^{-(t+r)m} \ \le \ (1-q^{-1}) q^{-s/d(f)}
\end{equation}
and when $s \equiv 0$ mod $N$,
\begin{equation}\label{I1-infinite}
(1-q^{-1})^2 q^{-s/d(f)} \ \le \ I_1.
\end{equation}

For $I_2$ we would like to make a change of variables
$u = \phi(x)$ in $x$ so that $u^t = x$. This will
require some care if $t\ge 2$ but if we can do this, then the
idea is to make the
change of variables $y = u^r z$ in the $y$ integral
which will successfully separate the variables
by the quasi-homogeneity of $f$.
Of course this can be done easily if $t=1$ (that is,
when the ratio $\kappa_2/\kappa_1$ of the dilation
parameters is an integer) which
includes the homogeneous case. Some care needs to be
taken then $\kappa_2/\kappa_1 \notin {\Bbb N}$ or $t\ge 2$.

\subsubsection{Estimates for $I_2$: the case when $\kappa_2/\kappa_1
\in {\Bbb N}$}
To get an idea of where we are heading, let us consider
the treatment of $I_2$ when
$t=1$ (or equivalently, when $\kappa_2/\kappa_1 \in {\Bbb N}$); here
$\phi(x) = x$ above so there is no initial change
of variables in $x$. We will then discuss the modifications needed
to treat the general case.
When $t=1$, we proceed directly
to the second change of variables $y = x^r z$ in the inner
$y$ integral and write $I_2 = $
\begin{equation}\label{I2-homogeneous}
\sum_{
          m N \le s - 1
          }
q^{-(1+r)m} \int_{|x|=1} \left[\int_{|z|=1} {\mathcal C}(\pi^{-s +
mN} x^N h(z)) \, d\mu(z)\right] d\mu(x)
\end{equation}
where $h(z) = a z^{\beta} \prod_{j\ge 1} (z-\zeta_j)^{n_j}$.
We now interchange the order of integration and decompose the $z$
integral depending on the size of $h(z)$;
$$
I_2 = \sum_{
          m N \le s - 1
          }
q^{-(1+r)m} \sum_{\ell\ge 0} \int\limits_{\bigl\{|z|=1:
|h(z)|=q^{-\ell}\bigr\}}d\mu(z)\int\limits_{|x|=1} {\mathcal
C}(\pi^{-s + mN + \ell} [\pi^{-\ell} h(z)] x^N) d\mu(x).
$$
By property (C2) the $\ell$ sum vanishes for $\ell \le s - mN - 2$
and by property (C1) the inner $x$ integral equals to $(1-q^{-1})$
for $\ell \ge s - mN$. Hence we split $I_2 = I_{2,1} + I_{2,2}$
where
\begin{equation}\label{I21-sublevel}
I_{2,1} \ = \ (1-q^{-1}) \, \sum_{
          m N \le s - 1
          }
q^{-(1+r)m} \sum_{\ell\ge s-mN} \mu\bigl(\bigl\{|z|=1:
|h(z)|=q^{-\ell}\bigr\}\bigr)
\end{equation}
and
$$
I_{2,2} = \sum_{
          m N \le s - 1
          }
q^{-(1+r)m} \int_{\bigl\{|z|=1: |h(z)|=q^{-\ell_{m,s}}\bigr\}}
F_{m,s}(z) d\mu(z)
$$
where $\ell_{m,s} = s - mN - 1$ and $F_{m,s}(z) = \int_{|x|=1}
{\mathcal C}(\pi^{-1} [\pi^{-\ell_{m,s}} h(z)] x^N) d\mu(x)$.
If ${\mathcal C}$ satisfies the extra cancellation
condition (C3)' (which is the case for the problem of
polynomial congruences), then $I_{2,2} = 0$.
Property (C3) implies $|F_{m,s}(z)| \le C q^{-1/2}$ for $z$
satisfying $|h(z)| = q^{-\ell_{m,s}}$ and so
\begin{equation}\label{I22-estimate}
|I_{2,2}| \ \le \ C q^{-1/2} \, \sum_{
          m N \le s - 1}
q^{-(1+r)m} \mu\bigl(\bigl\{|z|=1:
|h(z)|=q^{-\ell_{m,s}}\bigr\}\bigr).
\end{equation}
From \eqref{I21-sublevel} and \eqref{I22-estimate} we see that we
need to understand the sets $\{|z|=1 : |h(z)| = q^{-\ell} \}$ and
here is where Proposition \ref{PS} comes into play.

Set ${\mathcal I}_{\mathfrak p} = \{j\ge 1 : \, \zeta_j \in {\bar
K} \}$. Recall that ${\bar K} = {\bar K}_{\mathfrak p}$ is the
completion of $K$ with respect to the valuation
$|\cdot|_{\mathfrak p}$. The set ${\mathcal
I}_{\mathfrak p}$ may be empty. From Proposition \ref{PS} and the subsequent remarks,
we see
that for any $\ell\ge 1$,
\begin{equation}\label{PS-used}
\bigl\{|z|=1 : |h(z)| = q^{-\ell} \bigr\} \ = \ \bigcup_{j\in
{\mathcal I}_{{\mathfrak p},\ell}} \bigl\{ |z|=1 : |z - \zeta_j|
= q^{-\ell/n_j} \bigr\}
\end{equation}
where ${\mathcal I}_{{\mathfrak p},\ell} = \{j\in {\mathcal I_{\mathfrak p}} : n_j \, | \ \ell \}$.
Since the union above is disjoint
(recall that $|\zeta_j - \zeta_k| = 1$ for all $j\not= k$ when
${\mathfrak p} \notin {\mathcal P}(f)$), we have
$$
I_{2,1} = (1-q^{-1})^2  \sum_{j\in {\mathcal I}_{\mathfrak p}}
\sum_{
          m N \le s - 1 }
q^{-(1+r)m} \sum_{n\ge (s-mN)/n_j} q^{-n} := (1-q^{-1})^2
\sum_{j\in {\mathcal I}_{\mathfrak p}} I_{2,1}^j.
$$
If ${\mathcal I}_{\mathfrak p} = \emptyset$, then $I_{2,1} = 0$
but if ${\mathcal I}_{\mathfrak p} \not= \emptyset$, the following
estimates hold. Since
$$
q^{-(1+r)m} \sum_{n \ge (s-mN)/n_j} q^{-n} \ \le \ C \,
q^{-s/n_j} q^{m (1+r)(d(f) - n_j)/n_j}
$$
for each $j\in {\mathcal I}_{\mathfrak p}$ and some absolute
constant $C$, we have
\begin{equation}\label{I21j-estimate}
I_{2,1}^j \ \le \ C \, q^{-s/n_j} \ \ \ \ I_{2,1}^j \ \le \ C \, s
q^{-s/d(f)} \ \ \ {\rm or} \ \ I_{2,1}^j \ \le \ C \,
q^{-s/d(f)}
\end{equation}
whenever $d(f) < n_j , \ d(f) = n_j $ or $n_j < d(f)$,
respectively. Furthermore
\begin{equation}\label{I21j-below}
q^{-s/n_j} q^{-1} \ \le \ I_{2,1}^j \ \ \ \ {\rm or} \ \ \ \ c \,
s q^{-s/d(f)} q^{-1} \ \le \ I_{2,1}^j
\end{equation}
for all $s\ge 1$ whenever $d(f) < n_j$ or $d(f) = n_j$,
respectively; however when $s\equiv 0$ mod $n_j$,
\eqref{I21j-below} improves to
\begin{equation}\label{I21j-infinite}
q^{-s/n_j} \ \le \ I_{2,1}^j \ \ \ \ {\rm or} \ \ \ \ c \, s
q^{-s/d(f)} \ \le \ I_{2,1}^j
\end{equation}
whenever $d(f) < n_j$ or $d(f) = n_j$, respectively.

We turn now to bounding $I_{2,2}$ and here
$\ell_{m,s} \ge 1$ unless $mN = s-1$ (which can happen only if $N|
\ s-1$) in which case $\ell_{m,s} = 0$. The complimentary case
$\ell = 0$ to \eqref{PS-used} is
$$
\bigl\{|z|=1 : |h(z)| = 1 \bigr\} \ = \ \bigl\{|z|=1 :
|z - \zeta_j| = 1 \ \ {\rm for \ all} \ j \bigr\}.
$$
In fact, Lemma \ref{krasner} implies that
$|z-\zeta_j| = 1$ automatically holds for all $z\in {\bar{\mathfrak o}}$ whenever
$\zeta_j \notin {\bar K}$. The right hand side is $\{z : |z|=1 \}$
when ${\mathcal I}_{\mathfrak p} = \emptyset$.
Therefore
$$
I_{2,2} \ = \ \sum_{mN\le s-2}  \sum_{j\in
{\mathcal I}_{\mathfrak p}^{m,s}} I_{2,2}^{m,j} \ + \
I_{2,2}^{*} \ \
:= \ \ \sum_{j\in {\mathcal I}_{\mathfrak p}} I_{2,2}^j
\ + \ I_{2,2}^{*}
$$
where ${\mathcal I}_{\mathfrak p}^{m,s} := \{ j\in
{\mathcal I}_{\mathfrak p} : n_j \, | \ \ell_{m,s}\}$,
$$
I_{2,2}^{m,j} =  q^{-(1+r)m}
\int_{\bigl\{|z|=1 \ : \ |z-\zeta_j| = q^{-\ell_{m,s}/n_j} \bigr\}} F_{m,s}(z) \ d\mu(z), \ \ \ \
I_{2,2}^j = \sum_{
\begin{array}{c}\scriptstyle
mN \le s - 2\\
         \vspace{-5pt}\scriptstyle n_j | \,\ell_{m,s}
         \end{array} }
I_{2,2}^{m,j}
$$
and
$$
I_{2,2}^{*} \ = \
q^{-(1+r)(s-1)/N}
\int_{\{|z|=1 \ : \ |z-\zeta_j| = 1, \, \forall j \}}
F_{(s-1)/N,s}(z) \ d\mu(z) .
$$
The term $I_{2,2}^{*}$ appears only if $s \equiv 1$ mod $N$
and if this is the case, then property (C3) implies that
\begin{equation}\label{I22-empty}
|I_{2,2}^{*}| \ \le \
C q^{-1/2} q^{-(1+r)(s-1)/N} \ = \ C q^{-[\frac{1}{2} - \frac{1}{d(f)}]}
q^{-s/d(f)}.
\end{equation}
We note $j\in {\mathcal I}_{\mathfrak p}^{m,s}$ implies
$n_j | \ s-1-mN$ which in turn implies ${\rm gcd}(n_j,N) | \, (s-1)$.
Therefore if ${\rm gcd}(n_j,N) \ge 2$ and
$s\equiv 0$ mod ${\rm gcd}(n_j,N)$, then $j\notin
{\mathcal I}_{\mathfrak p}^{m,s}$ for any $m\ge 0$.

Again, using property (C3) to estimate $I_{2,2}^{m,j}$,
we see that
$$
|I_{2,2}^{m,j}| \ \le \ C q^{-[\frac{1}{2} - \frac{1}{n_j}]}
q^{-\frac{s}{n_j}} q^{-m\frac{1+r}{n_j}[n_j - d(f)]}
$$
and so $|I_{2,2}^j| \ \le$
\begin{equation}\label{I22-improvement}
C q^{-[\frac{1}{2} - \frac{1}{n_j}]}
q^{-s/n_j}, \ \ C s q^{-[\frac{1}{2} - \frac{1}{d(f)}]}
q^{-s/d(f)}, \ \ {\rm or} \
C q^{-[\frac{1}{2} + \frac{1}{n_j} - \frac{2}{d(f)}]}
q^{-s/d(f)}
\end{equation}
depending on whether $d(f) < n_j, \ d(f) = n_j$
or $n_j < d(f)$, respectively. Recall that $I_{2,2} = 0$
for polynomial congruences and so it is only when treating the
character sums ${\mathcal S}_{\chi}$ that $I_{2,2}$ arises and in
this case we are assuming $d(f)\ge 2$ in this section.
In particular we have $|I_{2,2}| \le C q^{-s/d(f)}$
when ${\mathcal I}_{\mathfrak p} = \emptyset$ by \eqref{I22-empty}
and the estimates
\begin{equation}\label{I22-estimate2}
|I_{2,2}^j| \le C q^{-s/n_j}, \ \ |I_{2,2}^j| \le C s q^{-s/d(f)}
\ \ {\rm or} \ \
|I_{2,2}^j| \le  C q^{-s/d(f)}
\end{equation}
for $j\in {\mathcal I}_{\mathfrak p}$
if $d(f)<n_j, \ d(f) = n_j$ or $n_j < d(f)$, respectively.

\subsubsection{Estimates for $I_2$: the general case}
Here we describe the modifications needed when
$\kappa_2/\kappa_1 \notin {\Bbb N}$ or $t\ge 2$.
In the end we will arrive at the same estimates
\eqref{I21j-estimate}, \eqref{I21j-below} and \eqref{I21j-infinite}
for $I_{2,1}$ and \eqref{I22-estimate2} for $I_{2,2}$
for general $t\ge 1$.
Once we have succeeded in the initial change of variables
indicated at the outset of Section \ref{I}, the argument
for the estimates for $I_{2,1}$ are the same as in the
case $t=1$. However the argument to establish estimate \eqref{I22-estimate2} for general $t\ge 1$
will differ slightly from the case $t=1$ described above.

We write
$$
I_2 \ = \ \sum_{
          m N \le s - 1
          }
q^{-(t+r)m} \int_{|x|=1} F(x) \, d\mu(x)
$$
where
$$ F(x) \ := \ \int_{|y|=1} {\mathcal
C}(\pi^{-s + mN} f(x,y)) \, d\mu(y).
$$
We would like to make a change of variables $u = \phi(x)$
so that $u^t = x$. In order to carry this out, we fix
a generator $g$ of the multiplicative cyclic group
$G := {\Bbb F}_q \setminus \{0\}$
of nonzero elements of our underlying finite field
${\Bbb F}_q = {\bar{\mathfrak o}}/\pi{\bar{\mathfrak o}}$
with $q = p^f$ elements. Set $d := {\rm gcd}(t,q-1)$
and recall that the $t$th powers of $G$ are given by
$G^t = \{ g^d, g^{2d}, \ldots,
g^{(q-1)/d \cdot d} = g^{q-1} = 1\}$. Furthermore set
$$
D \ := \ \{|x| = 1 : x = x_0 + x_1 \pi + x_2 \pi^2 + \cdots, \ \ \
x_0 \in G^t \}
$$
so that $\{|x| =1 \}$ has the decomposition
$$
\{ x \in {\bar{\mathfrak o}} : |x| = 1 \} \ = \
D \cup g D \cup g^2 D \cdots\cdots \cup g^{d-1} D \ = \
\bigcup_{e=0}^{d-1} g^e D
$$
into $d$ disjoint open sets. Therefore we can write
$$
\int_{|x|=1} F(x) \, d\mu(x) \ = \ \sum_{e=0}^{d-1} \int_{g^e D} F(x) \, d\mu(x) \ = \ \sum_{e=0}^{d-1} \int_D F(g^e x) d\mu(x).
$$
For each $x = x_0 + x_1 \pi + \cdots \in D$, there are precisely
$d$ solutions $u_0 \in {\Bbb F}_q \setminus \{0\}$ to
$u_0^t = x_0$, and by Hensel's lemma (note that if the characteristic
of ${\Bbb F}_q$ is positive, it is larger than $t$ by hypothesis) each such solution lifts uniquely to a solution
$u \in {\bar{\mathfrak o}}$ of $u^t = x$. We single out the
solution corresponding to $u_0 = g^{\theta}$ with
$0\le \theta \le (q-1)/d - 1$. This defines an analytic
isomorphism $\phi: D \to \phi(D)$ so that if $u = \phi(x)$,
then $u^t = x$. Therefore we can make the change of variables
$u = \phi(x)$ (see \cite{I-2} for a general change of variables
formula in our setting)
in each of the $d$ integrals above,
$$
\int_{|x|=1} F(x) \, d\mu(x) = \sum_{e=0}^{d-1} |t\cdot{\bf 1}|
\int_{\phi(D)} F(g^e u^t) \, d\mu(u)
$$
so that (we throw in $t\cdot{\bf 1}$ into our collection
${\mathcal A}$ in Section \ref{exceptional} to ensure
that $|t\cdot{\bf 1}| = |t\cdot{\bf 1}|_{\mathfrak p} = 1$ for all ${\mathfrak p}\notin {\mathcal P}(f)$)
$$
I_2 \ = \ \sum_{e=0}^{d-1} \sum_{
          m N \le s - 1
          }
q^{-(t+r)m} \int_{\phi(D)} F(g^e u^t ) \, d\mu(u).
$$
The function $F(x)$ is an integral in $y$ and
as before we make the change of variables $z = u^r y$
in the $y$ integral which brings us to the
analogue of \eqref{I2-homogeneous} for general $t$:
$$
I_2 \ = \  \sum_{e=0}^{d-1} \sum_{
          m N \le s - 1
          }
q^{-(1+r)m} \int_{\phi(D)} \left[\int_{|z|=1} {\mathcal C}(\pi^{-s +
mN} h_e(z) u^N ) \, d\mu(z)\right] d\mu(u)
$$
where $h_e(z) = a z^{\beta} \prod_{j=1}^M (z -  \zeta_j g^{er})^{n_j}$.
We now proceed exactly as in the case $t=1$, interchanging
the order of integration and decomposing the $z$ integral
according to the size of $h_e(z)$, etc... The main difference
is that the $x$ integral over the set $\{|x|=1\}$ has now
been replaced by a $u$ integral over the set $\phi(D)$. From
the definition of $\phi$ and $D$ we see that
$$
\phi(D) \ = \ \bigcup_{\theta = 0}^{\frac{q-1}{d} - 1}
B_{q^{-1}}(u_{\theta}) \ = \ \bigcup_{\theta = 0}^{\frac{q-1}{d} - 1}
\{u\in {\bar{\mathfrak o}} : |u - u_{\theta}|\le q^{-1} \}
$$
is a disjoint union of $(q-1)/d$ balls where $u_{\theta} = g^{\theta}$.
Hence $\mu(\phi(D)) = (1 - q^{-1})/d$. For a fixed $0\le e \le d-1$
and $m N \le s-1$, we need to understand $\sum_{\ell\ge 0}
H_{m,e}^{\ell}$
where
$$
H_{m,e}^{\ell} \ := \
\int_{\{|z|=1: |h_e(z)|=q^{-\ell}\}} d\mu(z)
\int_{\phi(D)} {\mathcal C}(\pi^{-s+mN+\ell} [\pi^{-\ell}h_e(z)]
u^N) \, d\mu(u)
$$
so that
$$
I_2 \ = \ \sum_{e=0}^{d-1} \sum_{mN\le s-1} q^{-(r+t)m}
\sum_{\ell\ge 0} H_{m,e}^{\ell}.
$$
As mentioned earlier there is a slight strengthening of property (C2),
namely that integration over $\{|x|=1\}$ can be replaced
by any finite union of disjoint balls $B_{q^{-1}}$ with
radius $q^{-1}$, which holds for our ${\mathcal C} = \psi$ and
${\mathcal C} = {\bf 1}_{\bar{\mathfrak o}}$. Therefore as before,
the $\ell$ sum above vanishes when $\ell + mN - s \le -2$.
By property (C1) the inner $u$ integral
in equals to $(1-q^{-1})/d$
for $\ell \ge s - mN$. Hence, proceeding exactly
as in the $t=1$ case, we split $I_2 = I_{2,1}
 + I_{2,2}$
where
$$
I_{2,1} =
(1-q^{-1})/d \, \sum_{e=0}^{d-1} \sum_{mN\le s-1} q^{-(r+t)m}
\sum_{\ell\ge s-mN} \mu\bigl(\bigl\{|z|=1:
|h_e(z)|=q^{-\ell}\bigr\}\bigr)
$$
and $I_{2,2} =$
$$
\sum_{e=0}^{d-1} \sum_{mN\le s-1} q^{-(r+t)m} \int_{\bigl\{|z|=1: |h_e(z)|=q^{-\ell_{m,s}}\bigr\}}
\Bigl[\int_{\phi(D)}
{\mathcal C}(\pi^{-1} [\pi^{-\ell_{m,s}} h_e(z)] u^N) d\mu(u)\Bigr]
d\mu(z)
$$
where as before $\ell_{m,s} = s - mN - 1$.
If ${\mathcal C} = {\bf 1}_{\bar{\mathfrak o}}$,
the case for the problem of polynomial congruences,
then $I_{2,2} = 0$.

As in the case $t=1$, we see that the estimates
\eqref{I21j-estimate}, \eqref{I21j-below}
and \eqref{I21j-infinite}
for $I_{2,1}$ hold in the general case. The estimates
\eqref{I22-empty}, \eqref{I22-improvement} and
\eqref{I22-estimate2} for $I_{2,2}$, although true, require
a modified argument. We write
$I_{2,2} = \sum_{e=0}^{d-1} \sum_{mN\le s-1} q^{-(r+t)m} S_{e,m}$
where
$$
S_{e,m} \ = \
\int_{\bigl\{|z|=1: |h_e(z)|=q^{-\ell_{m,s}}\bigr\}}
\Bigl[\int_{\phi(D)}
{\mathcal C}(\pi^{-1} [\pi^{-\ell_{m,s}} h_e(z)] u^N) d\mu(u)\Bigr]
d\mu(z).
$$
Suppose first that $\ell_{s,m} = s - mN -1 \ge 1$. Then
from \eqref{PS-used}, we have
$$
S_{e,m} =
\sum_{j\in
{\mathcal I}_{{\mathfrak p}}^{m,s}} \int_{\{|z|=1 : |z -
g^{er} \zeta_j| = q^{-\ell_{m,s}/n_j}\}}
\Bigl[\int_{\phi(D)}
{\mathcal C}(\pi^{-1} [\pi^{-\ell_{m,s}} h_e(z)] u^N) d\mu(u)\Bigr]
d\mu(z)
$$
where as before
${\mathcal I}_{{\mathfrak p}}^{m,s} = \{j\in {\mathcal I_{\mathfrak p}} : n_j \, | \ \ell_{m,s} \}$. For each $j \in
{\mathcal I}_{{\mathfrak p}}^{m,s}$, we make the change of
variables $w = \pi^{-\ell_{m,s}/n_j}(z-g^{er}\zeta_j)$ so that
$$
S_{e,m} =
\sum_{j\in
{\mathcal I}_{{\mathfrak p},m}}
q^{-\ell_{m,s}/n_j}
 \int_{|w|=1}
\Bigl[\int_{\phi(D)}
{\mathcal C}(\pi^{-1} K_e w^{n_j} u^N) d\mu(u)\Bigr]
d\mu(z)
$$
where $K_e = a [g^{er}\zeta_j]^{\beta}
\prod_{k\not= j} g^{ern_k}(\zeta_j - \zeta_k)^{n_k}$;
in particular $|K_e| = |K_e|_{\mathfrak p} = 1$ for
all ${\mathfrak p} \notin {\mathcal P}(f)$.
The assumption $\ell_{m,s}\ge 1$ was used here.

Therefore we can write
$$
I_{2,2} \ = \ \sum_{e=0}^{d-1}\sum_{mN\le s-2}  \sum_{j\in
{\mathcal I}_{\mathfrak p}^{m,s}} I_{e}^{m,j} \ + \
I_{2,2}^{*} \ \
:= \ \ \sum_{j\in {\mathcal I}_{\mathfrak p}} I_{2,2}^j
\ + \ I_{2,2}^{*}
$$
where
$$
I_{e}^{m,j} = q^{-(r+t)m-\ell_{m,s}/n_j}
\int_{\phi(D)}\Bigl[
\int_{|w|=1}
{\mathcal C}(\pi^{-1} K_e w^{n_j} u^N) d\mu(w)\Bigr]
d\mu(w)
$$
and
$$
I_{2,2}^{*} \ = \
q^{-(t+r)(s-1)/N}
\int_{\{|z|=1: |h_e(z)|= 1\}}
\Bigl[\int_{\phi(D)}
{\mathcal C}(\pi^{-1} [ h_e(z)] u^N) d\mu(u)\Bigr]
d\mu(z).
$$
The term $I_{2,2}^{*}$ appears only if $s \equiv 1$ mod $N$.
If ${\mathcal I}_{\mathfrak p} = \emptyset$,
then
$\{|z|=1 : |h_e(z)| = 1 \} = \{|z|=1\}$
and
if
${\mathcal I}_{\mathfrak p} \not= \emptyset$,
$$
\{|z|=1: |h_e(z)| = 1\} \ = \
\bigl\{z\in {\bar{\mathfrak o}} : |z|=1\bigr\} \setminus
\bigcup_{j\in {\mathcal I}_{\mathfrak p}} B_{q^{-1}}(z_{0,j})
$$
where $g^{er} \zeta_j = z_{0,j} + z_{1,j} \pi + \cdots$
(of course for each $j\in {\mathcal I}_{\mathfrak p}$,
$g^{er}\zeta_j \in {\bar{\mathfrak o}}$). Whether
${\mathcal I}_{\mathfrak p}$ is or is not empty, we interchange
the order of integration so that
$$
I_{2,2}^{*} = q^{-(t+r)(s-1)/N}
\Bigl[
 \int_{\phi(D)}
\Bigl[\int_{|z|\le 1}
{\mathcal C}(\pi^{-1} u^N h_{e}(z)) d\mu(z)\Bigr]
d\mu(u) \ + \ E\Bigr]
$$
where $|E| \le C q^{-1}$.
When ${\mathcal C} = \psi$, the integral $\int_{|z|\le 1}
\psi(\pi^{-1} u^N h_e(z)) d\mu(z)$ is a character sum
over a finite field for each fixed $u$ with $|u|=1$
and therefore we can appeal to A. Weil's work to obtain the bound
\begin{equation}\label{I22-empty-again}
|I_{2,2}^{*}| \ \le \ C q^{-[\frac{1}{2} - \frac{1}{d(f)}]}
q^{-s/d(f)}
\end{equation}
in this case.

Using property (C3) to estimate $I_{e}^{m,j}$,
we see that
$$
|I_{e}^{m,j}| \ \le \ C q^{-[\frac{1}{2} - \frac{1}{n_j}]}
q^{-\frac{s}{n_j}} q^{-m\frac{t+r}{n_j}[n_j - d(f)]}
$$
and so $|I_{2,2}^j| \ \le$
\begin{equation}\label{I22-improvement-again}
C q^{-[\frac{1}{2} - \frac{1}{n_j}]}
q^{-s/n_j}, \ \ C s q^{-[\frac{1}{2} - \frac{1}{d(f)}]}
q^{-s/d(f)}, \ \ {\rm or} \
C q^{-[\frac{1}{2} + \frac{1}{n_j} - \frac{2}{d(f)}]}
q^{-s/d(f)}
\end{equation}
depending on whether $d(f) < n_j, \ d(f) = n_j$
or $n_j < d(f)$, respectively. Recall that $I_{2,2} = 0$
for polynomial congruences and so it is only when treating the
character sums ${\mathcal S}_{\chi}$ that $I_{2,2}$ arises and in
this case we are assuming $d(f)\ge 2$ in this section.
In particular we have $|I_{2,2}^{*}| \le C q^{-s/d(f)}$
by \eqref{I22-empty-again}
and the estimates
\begin{equation}\label{I22-estimate2-again}
|I_{2,2}^j| \le C q^{-s/n_j}, \ \ |I_{2,2}^j| \le C s q^{-s/d(f)}
\ \ {\rm or} \ \
|I_{2,2}^j| \le  C q^{-s/d(f)}
\end{equation}
for $j\in {\mathcal I}_{\mathfrak p}$
if $d(f)<n_j, \ d(f) = n_j$ or $n_j < d(f)$, respectively.

Finally we observe that if $j \in {\mathcal I}_{{\mathfrak p}}^{m,s}$,
then $n_j \, | \ \ell_{m,s}$ and this implies that
${\rm gcd}(N,n_j) \, | \ s-1$. Therefore if
${\rm gcd}(n_j,N) \ge 2$ and
$s \equiv 0$ mod
${\rm gcd}(N, n_j)$, then
$j\notin {\mathcal I}_{{\mathfrak p}}^{m,s}$
for any $m\ge 0$.
This will be important when we turn our attention to
the lower bound \eqref{abstract-sum-infinite} in Theorem \ref{main-abstract}.

\subsubsection{Putting the estimates together for $I$}\label{I}
We combine the estimates derived above to give bounds
for
$$
I \ = \ I_1 + I_2 \ = \ I_1  \ +  \ \sum_{j\in {\mathcal I}_{\mathfrak p}} I^j_{2,1} \ + \ I_{2,2}.
$$
First we consider upper bounds for $I$ and
we begin by treating those $f$ not in any $E_m, \, m\ge 1$.
Then by Lemma \ref{IM-2} we see that
$m_K(f) \ge d(f)$ if and only if
there is some multiplicity $n_j \ge d(f)$ (and so
necessarily the multiplicity $n_j$ is
associated to a root $\zeta_j \in K$).
Therefore
the estimates \eqref{I1-above},
\eqref{I21j-estimate} (valid for general $t\ge 1$),
\eqref{I22-empty-again}
and \eqref{I22-estimate2-again} give the
desired bound for $I$ from above; when ${\mathfrak p}\notin {\mathcal P}(f)$,
\begin{equation}\label{I-conj-above}
|I| \ \le \ C s^{\nu(f)} q^{-s/h(f)}
\end{equation}
where $C=C_{deg(f)}$.
This holds for general $f$
if ${\mathcal C} = {\bf 1}_{\bar{\mathfrak o}}$
and for $f$ with $d(f)\ge 2$ if
${\mathcal C} = \psi$; recall that in this
section, when ${\mathcal C} = \psi$, we are assuming
that $2\le d(f)$  in which case
$i(f) = \nu(f)$ (see Section \ref{two-cases} for
the case of character sums when $d(f)<2$).

When $f\in E_m, \, m\ge 2$, there are two conjugate
roots $\zeta, \zeta^{*}$ of degree 2 over $K$, $m_K(f) = 0$,
$m = d(f) = h(f)$ and $n_1 = n_2 = m$. If $\zeta, \zeta^{*} \in {\bar K}$,
then $i_{\mathfrak p}(f) = \nu_{\mathfrak p}(f) = 1$
and ${\mathcal I}_{\mathfrak p} =
\{1,2\}$. On the other hand if $\zeta, \zeta^{*} \notin {\bar K}$,
then $i_{\mathfrak p}(f) = \nu_{\mathfrak p}(f) =0$, ${\mathcal I}_{\mathfrak p}
= \emptyset$ and so $I_{2,1} = 0$. Hence \eqref{I1-above},
\eqref{I21j-estimate},
\eqref{I22-empty-again} and \eqref{I22-estimate2-again}
show that
\begin{equation}\label{I-conj-exception}
|I| \ \le \ C s^{{\nu}_{\mathfrak p}(f)} q^{-s/h(f)}
\end{equation}
when ${\mathfrak p} \notin {\mathcal P}(f)$.
When $f\in E_1$, $m_K(f) = 0$ and $d(f) = 1$.
The estimates for the character sum $S_{\chi}$ are
treated in Section \ref{E1}; note that for
the character sum, $f \in E_1$ is not an exceptional
case and we obtain uniform bounds for all $s\ge 1$ and
${\mathfrak p} \notin {\mathcal P}(f)$. This leaves
establishing \eqref{I-conj-exception}
for polynomial congruences when $f \in E_1$. Recall
that $I_{2,2} = 0$ for polynomial congruences.
If $\zeta, \zeta^{*} \in {\bar K}$,
then $\nu_{\mathfrak p}(f) = 1$
and ${\mathcal I}_{\mathfrak p} =
\{1,2\}$. Therefore \eqref{I1-above} and
\eqref{I21j-estimate} give the desired estimate.
On the other hand if $\zeta, \zeta^{*} \notin {\bar K}$,
then $\nu_{\mathfrak p}(f) =0$, ${\mathcal I}_{\mathfrak p}
= \emptyset$ and so $I_{2,1} = 0$ implying $I_2 = 0$. Hence $I = I_1$
and so \eqref{I1-above}
alone gives the desired estimate.

Next we will show that for polynomial
congruences ${\mathcal C} = {\bf 1}_{\bar{\mathfrak o}}$,
the lower bound
\begin{equation}\label{I-poly-infinite}
c s^{\nu(f)} q^{-s/h(f)} \ \le \ I
\end{equation}
holds for infinitely many $s\ge 1$, for some $c = c_{deg(f)} > 0$
and ${\mathfrak p}\notin {\mathcal  P}(f)$
when $f$ does not lie in any $E_m$, $m\ge 1$.
Since ${\mathcal N}(f;\pi^s) \ge I$, the bound
\eqref{I-poly-infinite} establishes the
lower bound \eqref{abstract-poly-infinite}
in Theorem \ref{main-abstract}. To prove
\eqref{I-poly-infinite} we will restrict to
$s\equiv 0$ mod $N \prod_{j\in {\mathcal I}_{\mathfrak p}} n_j$ and
in particular $s \equiv 0$ mod $N$ so \eqref{I1-infinite}
implies $I \ge I_1 \ge c q^{-s/d(f)}$ for
these values of $s$. This proves \eqref{I-poly-infinite}
when $m_K(f) < d(f)$ in which case $\nu(f) = 0$
and $d(f) = h(f)$. When $m_K(f) \ge d(f)$,
Lemma \ref{IM-2} implies that there is a unique
root $\zeta_{j_{*}} \in K$ such that $n_{j_{*}}\ge d(f)$.
and so \eqref{I21j-infinite} implies
$I \ge I_{2,1}^{j_{*}} \ge c s^{\nu(f)} q^{-s/h(f)}$
whenever $s \equiv 0$ mod $n_{j_{*}}$. This
establishes \eqref{I-poly-infinite} and therefore
\eqref{abstract-poly-infinite} when $f \notin E_m$
for any $m\ge 1$. When $f\in E_m$ for some $m\ge 1$,
then $h(f) = d(f) = m = n_1 = n_2$.
If the conjugate roots $\zeta, \zeta_{*}$ do not
belong to ${\bar K}$, then $\nu_{\mathfrak p}(f) = 0$
and \eqref{I1-infinite} implies
$I  \ge I_1 \ge c q^{-s/d(f)}$ if $s \equiv 0$ mod $N$.
If the conjugate roots belong to ${\bar K}$, then
$\nu_{\mathfrak p}(f) = 1$, ${\mathcal I}_{\mathfrak p}
= \{1,2\}$ and \eqref{I21j-infinite} implies
$I \ge I_{2,1}^1 \ge c s q^{-s/d(f)}$
if $s \equiv 0$ mod $m$. Hence \eqref{I-poly-infinite}
(and hence \eqref{abstract-poly-infinite})
holds for $f\in E_m$ with $\nu(f)$ replaced
by $\nu_{\mathfrak p}(f)$.

We now show that when ${\mathcal C} = \psi$
(that is for character sums), $f$ is not linear
and $f \notin E_m$ for any $m\ge 2$,
\begin{equation}\label{I-sum-infinite}
c s^{i(f)} q^{-s/h(f)} \ \le \ |I|
\end{equation}
holds for infinitely many $s\ge 1$, for some $c = c_{deg(f)} > 0$
and ${\mathfrak p}\notin {\mathcal  P}(f)$.
Recall that in this section, we are assuming
that $d(f)\ge 2$ when
treating character sums; the case when $d(f) < 2$
is treated in Section \ref{two-cases}. Nevertheless
the analysis we give here will handle certain
situations when $d(f)<2$; more precisely,
the analysis will cover those $f$ with
$m_K(f) < d(f)$ (except for the case when
$f\in E_1$ which we treat separately in
Section \ref{E1}) and also those $f$ with
$d(f) \le m_K(f)$ but either $2 < m_K(f)$
or $2 = d(f) = m_K(f)$. This
will help alleviate
the analysis in Section \ref{two-cases}.

Note that $f(x,y) = a y + b x$ is linear if
and only if $N = t \alpha + r \beta + t r n = 1$
and so we may assume without
loss of generality that $N\ge 2$. Furthermore
we will restrict ourselves to
$s\equiv 0$ mod $N \prod_{j\in {\mathcal I}_{\mathfrak p}} n_j$
when establishing \eqref{I-sum-infinite} and so in
particular $s\not\equiv 1$ mod $N$ since $N\ge 2$
which implies that $I_{2,2}^{*} = 0$ for these
values of $s$.
We consider first the case $m_K(f)\ge d(f)$
(then $h(f) = m_K(f) \ge d(f)$) and so
by Lemma \ref{IM-2},
there is a unique
root $\zeta_{j_{*}} \in K$ with multiplicity
$n_{j_{*}} = h(f) \ge d(f)$.
Suppose first $n_{j_{*}} > 2$ in which case
we will use the improved bound \eqref{I22-improvement-again}
for $I_{2,2}^{j_{*}}$ so that for any $\epsilon>0$,
$|I_{2,2}^{j_{*}}| \le \epsilon s^{\nu(f)} q^{-s/h(f)}$
if $q$ is large enough and this is the case when
${\mathfrak p} \notin {\mathcal P}(f)$. For
other $j\in {\mathcal I}_{\mathfrak p}$ with $j\not= j_{*}$,
we have $n_j < d(f)$ and we will use the bound in \eqref{I22-improvement-again}
$$
|I_{2,2}^j| \ \le \ C
q^{-[\frac{1}{2} + \frac{1}{n_j} - \frac{2}{d(f)}]}
q^{-s/d(f)},
$$
noting that the exponent $\frac{1}{2} + \frac{1}{n_j} - \frac{2}{d(f)}$
is always strictly positive. This is certainly the case if
$d(f)\ge 2$ and if $d(f)<2$, then $n_j =1$ and so the exponent
still remains strictly positive as long as $d\ge 4/3$. We
claim that the existence of two multiplicities,
$n_j = 1 < d(f) \le n_{j_{*}}$ rules out the possibility
that $d(f) \le 4/3$. If $d(f) \le 4/3$, then
$$
3 r t \le tr (1 + n_{j_{*}}) \le t r n \le N =
d(f)(r + t) \le (4/3)(r + t)
$$
which implies that $9 r t \le 4 t + 4 r$ and this
is easily seen to be impossible. Therefore for any
$\epsilon>0$ and for any
$j\in {\mathcal I}_{\mathfrak p}$ with $j \not= j_{*}$,
we have the same estimate as for
$I_{2,2}^{j_{*}}$; namely,
$|I_{2,2}^j| \le \epsilon s^{i(f)} q^{-s/h(f)}$
for $q$ large enough. Finally since $s\equiv 0$ mod
$n_{j_{*}}$, \eqref{I21j-infinite} implies
$I_{2,1} \ge c s^{i(f)} q^{-s/h(f)}$ and so for $q$
large enough,
$$
|I| = |I_1 + I_{2,1} + I_{2,2}| \ge I_1 + I_{2,1} - \epsilon
s^{i(f)} q^{-s/h(f)} \ge (c/2) s^{i(f)} q^{-s/h(f)}
$$
since $I_1\ge 0$. This establishes \eqref{I-sum-infinite}
when $n_{j_{*}}=h(f) > 2$ and we turn now to the case
$n_{j_{*}} = m_K(f) = d(f) = 2$. In this case,
$N = d(f)(t+r) = 2 (t+r)$ and so ${\rm gcd}(n_{j_{*}},N) = 2$
implying $I_{2,2}^{j_{*}} = 0$; recall that
$$
I_{2,2}^{j_{*}} = \sum_{e=0}^{d-1}\sum_{
\begin{array}{c}\scriptstyle
mN \le s - 2\\
         \vspace{-5pt}\scriptstyle n_{j_{*}} | \,\ell_{m,s}
         \end{array} }
I_{e}^{m,j_{*}}
$$
and $n_{j_{*}} = 2 \not| \, \ell_{m,s} = s - m N - 1$
for any $m\ge 0$ since $N = 2 (t+r)$ and $s \equiv 0$ mod
$n_{j_{*}} = 2$. All other $j\in  {\mathcal I}_{\mathfrak p}$
with $j\not= j_{*}$ must satisfy $n_j = 1$ since
$n_j < d(f) = 2$ and so the analysis above shows
$|I_{2,2}| \le \epsilon s^{i(f)} q^{-s/h(f)}$ if $q$
is large enough which leads to the bound \eqref{I-sum-infinite}
as before.

We turn to the case $m_K(f) < d(f)$ and here,
again by Lemma \ref{IM-2} (recall that we are
not treating $f\in E_1$, see
Section \ref{E1} for this case), we have $n_j < d(f)$
for every $j\in {\mathcal I}_{\mathfrak p}$.
If $d(f)> 4/3$, the improved estimate \eqref{I22-improvement-again}
shows that for each $j\in {\mathcal I}_{\mathfrak p}$,
$|I_{2,2}^j| \le
\epsilon q^{-s/d(f)}$ for $q$ large enough.
If $d(f)\le 4/3$, then necessarily $\alpha = \beta = 0$
and either $n=1$ or $n=2$. If $n=1$, then
$f(x,y) = a(y^t - \zeta x^r)$ for some $\zeta \in K$
and the character sum $S_{\chi}(f;\pi^s)$ can be easily
evaluated when $s \equiv 0$ mod $N = r t$ (see
for example, Section \ref{two-cases}); in particular
one verifies that \eqref{I-sum-infinite} or more
generally the lower bound
\eqref{abstract-sum-infinite} in Theorem \ref{main-abstract}
in this case.
When $n=2$, then necessarily $t=1$, $r=2$, $N=4$, $d(f) = 4/3$
and so $f(x,y) = a (y - \zeta x^2) ( y - \eta x^2)$
for some $\zeta \not= \eta \in K$. In this case
we can tweak the argument above and improve upon the
estimate \eqref{I22-estimate2-again} for each $I_{2,2}^j$;
in fact, since $N=4$ and $s\equiv 0$ mod $N$ in
this case, the terms $mN = s-1, mN = s-2$ and $mN = s-3$
do not arise in the $m$ sum defining $I_{2,2}^j$ leading
to the improved bound
$$
|I_{2,2}^j| \ \le \ C q^{-[\frac{1}{2} + \frac{3}{n_j} -
\frac{4}{d(f)}]} q^{-s/d(f)}.
$$
Since $n_1 = n_2 = 1$ and $d(f) = 4/3$ in this case,
we again can conclude $|I_{2,2}^j| \le \epsilon q^{-s/d(f)}$
if $q$ is large enough. Altogether, when $m_K(f) < d(f)$
and $f\notin E_1$, we have
$|I_{2,2}| \le \epsilon q^{-s/d(f)}$ if $q$ is large enough
and so, since $I_1 \ge c q^{-s/d(f)}$ when $s \equiv 0$ mod $N$
by \eqref{I1-infinite},
$$
|I| = |I_1 + I_{2,1} + I_{2,2}| \ge I_1 + I_{2,1} -
\epsilon q^{-s/d(f)} \ge (c/2) q^{-s/d(f)} = (c/2) s^{i(f)}q^{-s/h(f)}
$$
since $I_{2,1} \ge 0$. This establishes \eqref{I-sum-infinite}
when $m_K(f) < d(f)$ and $f\notin E_1$.

When $f\in E_m$ for some $m\ge 2$, the estimate
\eqref{I-sum-infinite} holds with the exponent $i(f)$
replaced by $i_{\mathfrak p}(f)$.
When the conjugate pair $\zeta, \zeta^{*}$ lies outside
${\bar K}$, $i_{\mathfrak p}(f) = 0$
and ${\mathcal I}_{\mathfrak p} = \emptyset$. Hence
$I_{2,1} = I_{2,2} =0$ when $s$ is restricted
to $s\equiv 0$ mod $N$ and so $I = I_1 \ge c q^{-s/d(f)}$
by \eqref{I1-infinite} when $s\equiv 0$ mod $N$,
proving \eqref{I-sum-infinite} with
$i(f) = i_{\mathfrak p}(f)$ in this case.
When the conjugate pair belongs to ${\bar K}$,
then $i_{\mathfrak p}(f) = 1$
and ${\mathcal I}_{\mathfrak p} = \{1,2\}$
lists the two multiplicities $n_1 = n_2 = m = d(f) = h(f)$
corresponding to the roots $\zeta$ and $\zeta^{*}$.
The bound \eqref{I21j-infinite} shows that
$I_1 + I_{2,1} \ge I_{2,1} \ge c s q^{-s/h(f)}$
for $s\equiv 0$ mod $m$. Proceeding as in the
analysis above when $f\notin E_m$, we see that $I_{2,2} = 0$
when $s \equiv 0$ mod $N$ since
then $I_{2,2}^{*}=0$ and
$I_{2,2}^j = 0$ for both $j=1$ and $j=2$; in fact,
$n_1 = n_2 = m \ge 2$
never divides $\ell_{k,s} = s - kN - 1$ for any $k\ge 0$ if
$s\equiv 0$ mod $N$. In fact, writing $s = k_{*} N$
and noting $N = 2m$, we have $\ell_{k,s} = 2m(k_{*} - k) - 1$
and so $m$ does not divide $\ell_{k,s}$ for
any $k$ since $m\ge 2$. Therefore
$I = I_1 + I_{2,1} \ge c s q^{-s/h(f)}$ in this case,
establishing \eqref{I-sum-infinite} with
$i(f)$ replaced with $i_{\mathfrak p}(f)$.

We note that when $f\in E_1$, the multiplicities
$n_1, n_2$ of the roots $\zeta$ and $\zeta^{*}$ are equal
to 1 which divides $\ell_{k,s}$ for every $k$ and so
$I_{2,2}$ will give a nontrivial contribution to the
character sum ${\mathcal S}_{\chi}$. In fact the
contribution $I_{2,2}$ cancels exactly with $I_{2,1}$
for $f\in E_1$ when $s \equiv 0$ mod $N=2$ (see Section \ref{E1}).

%

\subsection{Estimates for $II$}\label{II}
Next we treat $II$ which we write as
$$
II \ = \   \sum_{
\begin{array}{c}\scriptstyle
k_1, k_2 \ge 0\\
         \vspace{-5pt}\scriptstyle t k_2 < r k_1
         \end{array} } q^{-k_1-k_2}
\int\!\!\!\int_{|x|= 1, |y|=1} {\mathcal C}(\pi^{-s + k_1 \alpha +
k_2(\beta + t n)} f_{k_1,k_2}(x,y)) \, d\mu(x) d\mu(y)
$$
where
$$
f_{k_1,k_2}(x,y) = a x^{\alpha}y^{\beta + t n} + c \pi^{rk_1 - tk_2}
x^{\alpha + r} y^{\beta + t(n -1)} + \cdots + b \pi^{n(rk_1 - tk_2)}
x^{\alpha + rn} y^{\beta}.
$$
By property (C2) we see that the above sum vanishes when
$k_1\alpha + k_2(\beta + tn) \le s-2$ and therefore $II = II_1 + II_2$
where
$$
II_1 \ = \ (1-q^{-1})^2 \sum_{
\begin{array}{c}\scriptstyle
tk_2 < r k_1 \\
         \vspace{-5pt}\scriptstyle s\le k_1\alpha + k_2(\beta+tn)
         \end{array} } q^{-k_1-k_2},
$$
using property (C1), and
$$
II_2 \ = \ \sum_{
\begin{array}{c}\scriptstyle
tk_2 < r k_1 \\
         \vspace{-5pt}\scriptstyle  k_1\alpha + k_2(\beta+tn) = s-1
         \end{array} } q^{-k_1-k_2}
\int\!\!\!\int_{|x|= 1, |y|=1} {\mathcal C}(\pi^{-1} a x^{\alpha}y^{\beta + tn}) \, d\mu(x) d\mu(y).
$$
When we turn to establish \eqref{abstract-sum-infinite}
in Theorem \ref{main-abstract} for character sums, we will ensure that
$s$ lies along the subsequence $s \equiv 0$ mod
${\rm gcd}(\alpha, \beta + t n)$ so that if
${\rm gcd}(\alpha, \beta + tn) \ge 2$, the sum
defining $II_2$ is empty for these values of $s$.
On the other hand, if ${\rm gcd}(\alpha,\beta + t n) =1$,
then the double integral above is $-(1-q^{-1})q^{-1}$;
see Section \ref{monomial} for this easy computation.
This gives a better bound than the $q^{-1/2}$ bound which
property (C3) gives and we will use this improvement for
\eqref{abstract-sum-infinite}.

In the summand defining $II_1$, we have the bounds
$$
(s - k_1 \alpha)/(\beta + t n) \ \le \ k_2 \ < \ (r/t) k_1
$$
and so we can preform the $k_2$ sum first to bound
\begin{equation}\label{II1-initial}
II_1 \ \le \ C q^{-s/(\beta + t n)} \sum_{
t s/N < k_1} q^{- k_1 [ 1 - \alpha/(\beta + tn)]} .
\end{equation}
We divide the analysis into three
cases: (A) $\alpha < d(f)$, (B) $\alpha = d(f)$
and (C) $\alpha > d(f)$. This division into three cases is
equivalent to the exponent $[1 - \alpha/(\beta + tn)]$
in \eqref{II1-initial} being positive, zero and negative,
respectively. By Lemma \ref{IM-2},
case (B) implies $m_K(f) = d(f)$ and case (C)
implies $m_K(f) > d(f)$. Therefore in these cases,
we have $\alpha = h(f)$.

For case (A), we use \eqref{II1-initial} to see that
\begin{equation}\label{II1-A}
II_1 \ \le \ C q^{-s/d(f)}
\end{equation}
for some constant $C>0$ depending only on the
degree of $f$. For cases (B) and (C), we divide
$II_1 = II_{1,1} + II_{1,2}$ into two parts by
splitting the $k_1$ sum,
$$
II_{1,1} \ =
\sum_{
\begin{array}{c}\scriptstyle
t k_2 < r k_1, \ k_1 \le s/\alpha \\
         \vspace{-5pt}\scriptstyle s\le k_2(\beta+tn) + k_1\alpha
         \end{array} } q^{- k_1 - k_2}, \ \ \ \
II_{1,2} \ =
\sum_{
\begin{array}{c}\scriptstyle
t k_2 < r k_1, \ s/\alpha \le k_1 \\
         \vspace{-5pt}\scriptstyle s\le k_2(\beta+tn) + k_1\alpha
         \end{array} } q^{- k_1 - k_2} .
$$
For $II_{1,2}$ in cases (B) and (C), we simply use the restrictions
$k_1 \ge s/\alpha$ and $k_2 \ge 0$ to obtain
\begin{equation}\label{II12}
II_{1,2} \ \le \ C q^{-s/\alpha} \ = \ C q^{-s/h(f)}.
\end{equation}
For $II_{1,1}$ in cases (B) and (C), we use \eqref{II1-initial}
to see that
\begin{equation}\label{II11}
II_{1,1} \ \le \ C s q^{-s/h(f)} \ \ \ \ {\rm and}
\ \ \ \
II_{1,1} \ \le \ C q^{-s/h(f)},
\end{equation}
respectively.


For $II_2$, if ${\mathcal C}$ satisfies property (C3)' (the case
of polynomial congruences), then $II_2 =  0$ and so
we need to bound $II_2$ only for character sums and
in this case we are assuming $d(f)\ge 2$.
Using property (C3) to bound the integral in $II_2$
by $q^{-1/2}$, we obtain
$$
|II_2| \ \le \ C q^{-1/2} q^{-(s-1)/(\beta + tn)}
\sum_{t s/ N \le k_1 \le (s-1)/\alpha}
q^{-k_1[1 - \alpha/(\beta + t n)]}
$$
and splitting the analysis into cases (A), (B) and (C)
as above, we conclude that if $h(f)\ge 2$
(which is implied by our underlying assumption
$d(f)\ge 2$),
\begin{equation}\label{II2}
|II_2| \le C q^{-s/d(f)}, \ \ \
|II_2| \le C s q^{-s/d(f)} \ \ {\rm and} \ \
|II_2| \le C q^{-s/\alpha},
\end{equation}
respectively;
in fact the initial estimate for $II_2$ implies
$|II_2| \le C q^{-[1/2 - 1/d(f)]} q^{-s/d(f)}$,
$|II_2| \le C q^{-[1/2 - 1/d(f)]} s q^{-s/d(f)}$
and $|II_2| \le q^{-[1/2 - 1/h(f)]} q^{-s/h(f)}$
in the respective cases (A), (B) and (C).
This shows that if $h(f)>2$, then for any
$\epsilon>0$,
\begin{equation}\label{II2-error}
|II_2| \ \le \ \epsilon s^{i(f)} q^{-s/h(f)}
\end{equation}
if $q$ is large enough and this is the case
when ${\mathfrak p} \notin {\mathcal P}(f)$.
We will use \eqref{II2-error} for the proof
of \eqref{abstract-sum-infinite} in Theorem
\ref{main-abstract} (recall that for polynomial
congruences, $II_2 = 0$).

Putting \eqref{II1-A}, \eqref{II12}, \eqref{II11}
and \eqref{II2} together gives us the favourable upper
bound
\begin{equation}\label{II-poly-above}
|II| \ \le \ C \, s^{\nu(f)} q^{-s/h(f)}
\end{equation}
for ${\mathcal C} = {\bf 1}_{\bar{\mathfrak o}}$,
polynomial congruences, and
\begin{equation}\label{II-sum-above}
|II| \ \le \ C \, s^{i(f)} q^{-s/h(f)}
\end{equation}
for ${\mathcal C} = \psi$, character sums,
assuming $d(f)\ge 2$.



By the remark following the definition
of $II_2$, if
${\rm gcd}(\alpha, \beta + t n) = 1$,
then the estimates \eqref{II2} improve to
\begin{equation}\label{II2-above-improvement}
|II_2| \ \le \ q^{-1+1/h(f)} s^{i(f)} q^{-s/h(f)}
\end{equation}
and so in this case, \eqref{II2-error} holds
if $h(f)>1$.

Recall that in Section
\ref{I}, we successfully bounded $|I|$
from below (for character sums) for
infinitely many $s\ge 1$ for any
nonlinear $f \notin E_1$ satisfying $m_K(f) < d(f)$
or $m_K(f) \ge d(f)$ such that either $m_K(f)  > 2$
or $m_K(f) = d(f) = 2$. To bring $II$ in line with
these results let us observe that $II_2$
satisfies \eqref{II2-error} for infinitely many $s\ge 1$
when $f\notin E_1$ and when $m_K(f) < d(f) \le 2$ holds,
or $m_K(f) > 2 \ge d(f)$ holds or  $m_K(f) = d(f) = 2$ holds.
Here we will restrict to those $s$ satisfying
$s \equiv 0$ mod ${\rm gcd}(\alpha,\beta + tn)$.

We consider two situations. First, suppose
that ${\rm gcd}(\alpha,\beta + tn) \ge 2$. Then
$II_2 = 0$ since $s\equiv 0$ mod ${\rm gcd}(\alpha,\beta + tn)$
and so \eqref{II2-error} is trivially satisfied
in this case.
Second, suppose that ${\rm gcd}(\alpha,\beta + tn) = 1$
in which case we can use \eqref{II2-above-improvement}
and reduce to the situation where $d(f) \le 1$. Hence
we need only consider the cases when $m_K(f) < d(f)\le 1$
and when $m_K(f) > 2 > 1 \ge d(f)$ and we will show
that these situations cannot arise if $f\notin E_1$.

If it were the case that $m_K(f) > 2 > 1 \ge d(f)$, then there
would be a unique
root $\zeta_{*} \in K$ with multiplicity $n_{*} \ge 3$ and
hence
$3rt / (t+r) \le  d(f)  \le  1$
which is clearly impossible.
If is were the case that $m_K(f) < d(f) \le 1$,
then $m_K(f)$ must be zero and so there must be
at least two nonzero roots $\zeta, \zeta_{*} \notin K$
and hence $n\ge 2$. But this implies
$(t\alpha + r\beta + nrt)/(t+r) = d(f) \le 1$
which is impossible unless $\alpha = \beta = 0$,
$t=r=1$ and $n=2$. However this is precisely
the case when $f\in E_1$.

\subsection{Estimates for III}\label{III}
The analysis for the term III is the same
as for II.
We write III as
$$
III \ = \   \sum_{
\begin{array}{c}\scriptstyle
k_1, k_2 \ge 0\\
         \vspace{-5pt}\scriptstyle r k_1 < t k_2
         \end{array} } q^{-k_1-k_2}
\int\!\!\!\int_{|x|= 1, |y|=1} {\mathcal C}(\pi^{-s + k_1( \alpha + rn) +
k_2\beta} g_{k_1,k_2}(x,y)) \, d\mu(x) d\mu(y)
$$
where
$$
g_{k_1,k_2}(x,y) = b x^{\alpha + rn}y^{\beta} + \cdots + a
 \pi^{n(t k_2 - r  k_1)}
x^{\alpha} y^{\beta + t n}.
$$
By property (C2) we see that the above sum vanishes when
$k_1(\alpha+rn) + k_2\beta \le s-2$ and therefore $III = III_1 + III_2$
where
$$
III_1 \ = \ (1-q^{-1})^2 \sum_{
\begin{array}{c}\scriptstyle
r k_1 < t k_2 \\
         \vspace{-5pt}\scriptstyle s\le k_1(\alpha+rn) + k_2\beta
         \end{array} } q^{-k_1-k_2},
$$
using property (C1), and
$$
III_2 \ = \ \sum_{
\begin{array}{c}\scriptstyle
r k_1 <  t k_2 \\
         \vspace{-5pt}\scriptstyle  k_1(\alpha +rn)+ k_2\beta = s-1
         \end{array} } q^{-k_1-k_2}
\int\!\!\!\int_{|x|= 1, |y|=1} {\mathcal C}(\pi^{-1} b x^{\alpha+rn}y^{\beta}) \, d\mu(x) d\mu(y).
$$
The same estimates for II hold for III
with the same proofs. So we will only
state them. The following estimates hold:
\begin{equation}\label{III-poly-above}
|III| \ \le \ C s^{\nu(f)} q^{-s/h(f)}
\end{equation}
for polynomial congruences and
\begin{equation}\label{III-sum-above}
|III| \ \le \ C s^{i(f)} q^{-s/h(f)}
\end{equation}
for character sums, assuming $d(f) \ge 2$.
For polynomial congruences, $III_2 = 0$ and
for character sums, we have for every $\epsilon>0$,
\begin{equation}\label{III2-error}
|III_2| \ \le \ \epsilon s^{i(f)} q^{-s/h(f)}
\end{equation}
if $q$ large enough whenever $m_K(f)<d(f)$
or $d(f)\le m_K(f)$ with $2<m_K(f)$ or $2 = m_K(f) = d(f)$.

\subsection{The upper bounds in Theorem \ref{main-abstract}}\label{upper}
Since
$$
\int\!\!\!\int_{{\bar{\mathfrak o}}\times
{\bar{\mathfrak o}}} {\mathcal C}(\pi^{-s} f(x,y)) \, d\mu(x) d\mu(y) \ = \ I + II + III,
$$
we combine the estimates \eqref{I-conj-above},
\eqref{II-poly-above} and  \eqref{III-poly-above}
for ${\mathcal C} = {\bf 1}_{\bar{\mathfrak o}}$
to see that the upper bound in \eqref{main-poly-est-abstract}
holds for general quasi-homogeneous $f\notin E_m$ for any $m\ge 1$,
with the appropriate modifications, the Varchenko
exponent $\nu(f)$ replaced by $\nu_{\mathfrak p}(f)$,
when $f\in E_m$.

When ${\mathcal C} = \psi$, we combine
the estimates \eqref{I-conj-above},
\eqref{II-sum-above} and  \eqref{III-sum-above}
for ${\mathcal C} = \psi$
to see that the bound \eqref{main-sum-est-abstract} in Theorem \ref{main-abstract} holds for $f$ with $d(f)\ge 2$,
with the appropriate
modifications when $f\in E_m, m\ge 2$ lies in one of the exceptional classes. The case $f\in E_1$ for character sums is treated
in Section \ref{E1} below.

\subsection{The lower bounds \eqref{abstract-poly-infinite}
and \eqref{abstract-sum-infinite} in
Theorem \ref{main-abstract}}\label{infinite}
As we observed earlier, for the problem of
polynomial congruences, since ${\mathcal N}(f;\pi^s) \ge I$,
the bound \eqref{abstract-poly-infinite} in Theorem \ref{main-abstract}
follows from \eqref{I-poly-infinite} whenever $f$ does
not belong to any exceptional class $E_m, m\ge 1$. Similarly,
when $f\in E_m$ for some $m\ge 1$, we have
$$
{\mathcal N}(f;\pi^s) \ge I \ge c s^{\nu_{\mathfrak p}(f)} q^{-s/h(f)}
$$
for infinitely many $s\ge 1$,

For the problem of character sums, we will show that
\eqref{abstract-sum-infinite} holds
for any nonlinear $f\notin E_m, m\ge 1$
whenever $m_K(f) < d(f)$
or whenever $d(f)\le m_K(f)$ and either $2<m_K(f)$ or $2 = m_K(f) = d(f)$. The case $f\in E_1$ is treated in
Section \ref{E1} and the remaining cases will be treated in Section \ref{two-cases}. The argument establishing \eqref{I-sum-infinite}
showed $I = P + E$ where $P\ge 0$,
$P \ge c s^{i(f)} q^{-s/h(f)} \ge 0$ and $|E| \le (1/2) P$
for infinitely many $s\ge 1$. Hence
$I + II + III  = [P + II_1 + III_1] + [E + II_2 + III_2]$
and by
\eqref{II2-error},
\eqref{III2-error} and the fact that
$II_1, III_1 \ge 0$, we see that for infinitely many
$s\ge 1$,
$$
|I + II + III| \ge P - |E| - |II_2| - |III_2| \ge
(c/2) \, s^{i(f)} q^{-s/h(f)} - (c/4) s^{i(f)} q^{-s/h(f)}
$$
if $q$ large enough. The same argument shows that
$$
{\mathcal S}_{\chi}(f;\pi^s) = |I + II + III| \ge c s^{i_{\mathfrak p}(f)} q^{-s/h(f)}
$$
for infinitely many $s\ge 1$ whenever $f\in E_m$ for
some $m\ge 2$.

\subsection{The lower bound in \eqref{main-poly-est-abstract}
of Theorem \ref{main-abstract}}\label{I+II+III}
To complete the proof of Theorem \ref{main-abstract}
for polynomial congruences, we need to
establish the lower bound in \eqref{main-poly-est-abstract},
with the appropriate modifications when $f\in E_m$ for
some $m\ge 1$.
This is a bound for the number of polynomial congruences
${\mathcal N}(f;\pi^s) = I + II + III$ so
$I = I_1 + I_{2,1}$, $II = II_1$ and $III = III_1$ and
of course each of the terms are nonnegative. In the cases
$m_K(f) > d(f)$ and $m_K(f) = d(f)$ (the
two cases where ${\mathcal I}_{\mathfrak p}$ is
necessarily nonempty), the lower
bound in \eqref{main-poly-est-abstract} follows from
\eqref{I21j-below}; in fact in these cases,
the factor $q^{-2}$ can be replaced
by $q^{-1}$. When $f\in E_m$ for some $m\ge 1$
where the conjugate roots $\zeta, \zeta^{*}$
lie in ${\bar K}_{\mathfrak p}$, then
$\nu_{\mathfrak p}(f) = 1$, ${\mathcal I}_{\mathfrak p}
= \{1,2\}$, and \eqref{I21j-below}
again implies ${\mathcal N}(f;\pi^s) \ge c s q^{-s/h(f)} q^{-1}$
in this case.

In order to establish \eqref{main-poly-est-abstract} in case $m_K(f) < d(f)$, we need a bound from below
for
$$
I_1 + II_1 + III_1 \ = \ (1-q^{-1})^2
\sum_{
\begin{array}{c}\scriptstyle
k_1, k_2 \ge 0 \\
         \vspace{-5pt}\scriptstyle
         s\le k_1\alpha + k_2\beta
         + \min(rk_1, t k_2) n
         \end{array} }
 q^{- k_1 - k_2}.
$$
We claim that the
uniform bound
\begin{equation}\label{I+II}
c q^{-s/d(f)} q^{-2} \ \le \ I_1 + II_1 + III_1
\end{equation}
holds which will complete the proof of
\eqref{main-poly-est-abstract} in Theorem \ref{main-abstract}
when $f\notin E_m$ for any $m\ge 1$.
Writing $s\ge 1$ as $s = N m_{*} + T$ for some integers
$m_{*}\ge 0$ and $0\le T \le N-1$,
we define the integer $0\le L < t$ so that $L-1 < (Tt)/N \le L$.
With $L$, we define the integers
$k_1^{*} := t m_{*}  + L$ and $k_2^{*}$ so that
$$
k_2^{*} - 1 \ < \ r m_{*}  +
\frac{T}{\beta + tn} - \frac{\alpha L}{\beta+tn}
\ \le \ k_2^{*} .
$$
One easily checks that the integer $k_2^{*}$ defined
above is nonnegative.
We consider two cases: when $ t k_2^{*} \le r k_1^{*}$
and when $ r k_1^{*} < t k_2^{*}$. In the first case,
we see from the definition $k_2^{*}$,
$s \le k_1^{*} \alpha + k_2^{*}(\beta + t n)$.
In the second case one checks
that $s \le k_1^{*}(\alpha + rn) + \beta k_2^{*}$ hold,
or equivalently,
$$
\frac{T - (\alpha + rn)L}{\beta} \ \le \
\frac{T - \alpha L}{\beta + t n}
$$
which boils down to $(T n)/N \le L$.
Hence in either case,
$$
I_1 + II_1 + III_1 \gtrsim q^{-k_1^{*} - k_2^{*}} \ge
q^{-s/d(f)} q^{T/d(f)} q^{- [T-\alpha L]/(\beta + tn) -L -1} \ge q^{-s/d(f)} q^{\alpha/(\beta + tn) - 2}
$$
which gives the bound in \eqref{I+II}. Here we used
$(T t)/ N \ge L-1$. When $f\in E_m$ for some $m\ge 1$
where the conjugate roots $\zeta, \zeta^{*}$ do not
lie in ${\bar K}_{\mathfrak p}$, then $\nu_{\mathfrak p}(f) = 0$,
${\mathcal I}_{\mathfrak p} = \emptyset$ and the bound
\eqref{I+II} implies that ${\mathcal N}(f;\pi^s) \ge c q^{-s/h(f)}
q^{-2}$ in this case.

Let us look now at the example
$f(x,y) = y^4 - 2 x^6 \in {\Bbb Z}[X,Y]$ mentioned
after the statement of Theorem \ref{main-cong}. Here
$h(f) = 12/5$, $\nu(f) = 0$ and when we restrict to
$s\equiv 1$ mod 12, we have
$$
I_1 + II_1 + III_1 = (1-q^{-1})^2
\sum_{s\le 2\min(2k_1,k_2)} q^{-k_1-k_2}
\le c q^{-5s/12} q^{-19/12} .
$$
Furthermore if $p \equiv 3$ or $5$ mod $8$,
then $\pm \sqrt{2} \notin {\Bbb Q}_p$ and
so $I_{2,1} = 0$. Therefore for these values of $p$,
$$
{\mathcal N}(f;p^s) = I + II + III = I_1 + I_{2,1}
+ II_1 + III_1 = I_1 + II_1 + III_1
$$
and so  ${\mathcal N}(f;p^s) \le c q^{-s/h(f)} q^{-19/12}$
when $s\equiv 1$ mod 12. This illustrates that
we cannot replace the factor $q^{-2}$ with $q^{-1}$
in the lower bound \eqref{basic-cong} or
\eqref{main-poly-est-abstract}
of Theorems \ref{main-cong} and \ref{main-abstract}.

\subsection{Estimates for $f \in E_1$}\label{E1}
Here we treat separately the case of character
sums ${\mathcal S}_{\chi}(f;\pi^s)$ when $f\in E_1$; that is,
when $f(x,y) = a (y - \zeta x)(y - \zeta^{*}x)$
where $\zeta, \zeta^{*}$ are conjugate roots of degree
2 over $K$. For such $f$, $h(f) = d(f) = 1$ and the claimed estimates
for ${\mathcal N}(f;\pi^s)$ in Theorem \ref{main-abstract}
have already been established; namely,
$$c s q^{-s} q^{-2} \ \le \ {\mathcal N}(f;\pi^s) \ \le
C s q^{-s}
$$
and $c s q^{-s} \le {\mathcal N}(f;\pi^s)$ when $s\equiv 0$
mod 2.

The estimates for character sums
${\mathcal S}(f;\pi^s)$
when $f\in E_1$ are
different from those for polynomial congruences;
the uniform upper bound \eqref{main-sum-est-abstract} in Theorem \ref{main-abstract} is
\begin{equation}\label{E1-sum-upper}
|S_{\chi}(f; \pi^s)| \ \le \ C q^{-s}
\end{equation}
whenever ${\mathfrak p}\notin {\mathcal P}(f)$. Furthermore
the estimate \eqref{abstract-sum-infinite} reads
that for infinitely many $s\ge 1$,
\begin{equation}\label{E1-sum-infinite}
c q^{-s} \ \le \ |S_{\chi}(f;\pi^s)|
\end{equation}
holds whenever ${\mathfrak p}\notin {\mathcal P}(f)$.
The upper bound \eqref{E1-sum-upper} follows from
the work of Denef and Sperber \cite{DS} since
$f$ is nondegenerate with respect to its Newton
diagram (see also
the work of Cluckers \cite{C} for the abstract setting
of general local fields). Strictly speaking the
estimate \eqref{E1-sum-infinite}
does not follow from the work of
Denef and Sperber since the vertices $\{(0,2),(2,0)\}$
of the Newton polygon of $f$ lie in $\{0,1,2\}^2$.
Nevertheless we can see that \eqref{E1-sum-infinite}
holds from our analysis above. Recall our basic decomposition
$S_{\chi}(f;\pi^s) = I + II + III$ where $I = I_1 + I_2$,
$II = II_1 + II_2$ and $III = III_1 + III_2$; furthermore,
$II_2 = 0$ when $s \equiv 0$ mod ${\rm gcd}(\alpha, \beta + tn)$
and $III_2 = 0$ when $s \equiv 0$ mod ${\rm gcd}(\beta,\alpha + rn)$ (here ${\rm gcd}(\alpha, \beta + tn) = {\rm gcd}(\beta,\alpha + rn)
=2$ in our case $f\in E_1$). Due to the nondegeneracy of $f$,
we also have $I_2 = 0$ if $s\equiv 0$ mod 2. This follows
by the same argument establishing property (C2) for character sums,
adapted to double sums; in fact if $s\equiv 0$ mod 2, then
$\sigma := s - 2m \ge 2$ in the $m$ sum defining $I_2$ and
so if we write ${\vec u} := (x,z) \le \pi^{\sigma}{\bar{\mathfrak o}}$
(using our shorthand notation introduced in section \ref{completion})
as ${\vec u} = {\vec v} + \pi^{\sigma -1} {\vec w}$
with ${\vec v} \le \pi^{\sigma -1}{\bar{\mathfrak o}}$
and ${\vec w} \le \pi{\bar{\mathfrak o}}$, then
$\phi({\vec u}) \equiv \phi({\vec v}) + \pi^{\sigma -1}
\nabla \phi({\vec v})
\cdot {\vec w}$ mod
$\pi^{2\sigma - 2}{\bar{\mathfrak o}}$ which in turn
is equivalent mod $\pi^{\sigma}$ since $\sigma\ge 2$
(here $\phi(x,z) = h(z) x^2$). Also $\pi \not| \
{\vec u}$ is equivalent to $\pi \not| \ {\vec v}$ and
when this happens, $\pi \not| \ \nabla\phi({\vec v})$.
Therefore
$$
I_2 = \sum_{\sigma\ge 2} q^{-2(s-\sigma)} q^{-2\sigma} \sum_{
\begin{array}{c}\scriptstyle
{\vec v} \le \pi^{\sigma-1}{\bar{\mathfrak o}} \\
         \vspace{-5pt}\scriptstyle  \pi \not| \ {\vec v}
         \end{array} } \chi'(\pi^{-\sigma} \phi({\vec{v}}))
\sum_{
{\vec w} \le \pi{\bar{\mathfrak o}}}
\chi'(\pi^{-1}\nabla\phi({\vec v})\cdot {\vec w}) \ = \ 0
$$
as in the verification of property (C2) for character sums.
Since $II_1, III_1 \ge 0$, \eqref{E1-sum-infinite}
follows from \eqref{I1-infinite} which holds for $s\equiv 0$
mod 2 since $N=2$ in this case.

\section{The case $d(f) < 2$ for character sums}\label{two-cases}
Here we consider the character sums (or oscillatory integrals, see
\eqref{osc-S-chi})
$$
{\mathcal S}_{\chi}(f;{\mathfrak p}^s) \ = \ \int\!\!\!\int_{{\bar{\mathfrak o}}
\times{\bar{\mathfrak o}}} \psi(\pi^{-s} f(x,y)) d\mu(x) d\mu(y)
$$
when $d(f)<2$ and when $f$ consists of more than
one monomial.
We note that in this case the exceptional classes $E_m$ for $m\ge 2$
do not arise. In fact
if $f\in E_m$, then $d(f) = m, \, m_K(f)=0$ and $h(f)=m$.
If furthermore $d(f)<2$, then this forces
$f\in E_1$ and
such an $f$ does {\it not} belong to the exceptional class
for the character sum estimate \eqref{main-sum-est-abstract}
in Theorem \ref{main-abstract}. However such an $f$ {\it does} belong
to the exceptional class for the
polynomial congruences estimate \eqref{main-poly-est-abstract}.
In this case, the Varchenko
exponent $\nu(f) = \nu_{\mathfrak p}(f)$ depends
on the prime ideal ${\mathfrak p}$ as described in Theorem
\ref{main-abstract}.

To be quite specific, our goal here is to establish
the estimates \eqref{main-sum-est-abstract} and
\eqref{abstract-sum-infinite} in Theorem \ref{main-abstract}
when $d(f)<2$. In fact
we
need only establish \eqref{abstract-sum-infinite}
when $f\notin E_1$ (this case was already treated in
Section \ref{E1}), when $f$ is not linear and when $d(f) \le m_K(f) \le 2$
with $d(f)<2$; see Section \ref{infinite}.

We observe that when $d(f)<2$ the exponent
$i(f)$ is equal to zero even if a vertex of the Newton diagram
of $f$ lies on the bisectrix. In fact if $h(f)<2$, then
$i(f)=0$ by definition and if $d(f) < 2 \le h(f)$, then $m_K(f) \not= d(f)$
and so again $i(f)=0$.


When $d(f)<2$,
the list of possibilities
for $f$ is small and in the subcase $h(f)<2$, it
turns out that $f$ is nondegenerate with respect
to its Newton diagram so we can
appeal to the work of Denef and Sperber \cite{DS} or
Cluckers \cite{C-2}
to establish
the estimate \eqref{main-sum-est-abstract} in this case
(alternatively we can follow the arguments in the
previous sections, noting improved
finite field character sums at the appropriate places).
Strictly speaking the estimates in \cite{DS} or \cite{C-2}
carry a linear factor of $s$ when a vertex of the Newton
diagram lies on the bisectrix. However we will see that when
$d(f)<2$ and $f$ is not a monomial, this
only happens if $f(x,y) = a x (y - \zeta x^r)$ for some
$r\ge 1$ or
$f(x,y) = a y (y- \zeta x)$ where $\zeta$ is nonzero and lies
in $K$ (these are the only such cases which
arise under our assumption $\kappa_1\le \kappa_2$;
in general we should swap $x$ and $y$ and include
$f(x,y) = a y(x - \zeta y^r)$ for any $r\ge 1$).

In these two cases $h(f) = d(f) = m_K(f) = 1$ and
(for $f(x,y) = a x(y-\zeta x^r)$, say) \ \
${\mathcal S}_{\chi}(f;{\mathfrak p}^s) = $
$$
\int_{|x|\le 1} \psi(\pi^{-s}a\zeta x^{r+1}) d\mu(x)
\int_{|y|\le 1} \psi(\pi^{-s} a x y) d\mu(y) = \int_{|x|\le q^{-s}}
\psi(\pi^{-s} a \zeta x^{r+1}) d\mu(x).
$$
The last integral equals $q^{-s}$ since
$|\pi^{-s} a \zeta x^{r+1}| \le q^{s}q^{-s(r+1)} \le 1$ when
$|x| \le q^{-s}$ and this
implies the claimed estimates \eqref{main-sum-est-abstract}
and \eqref{abstract-sum-infinite} in this case. A similar identity holds for $f(x,y) =
a y (y -\zeta  x)$.

We now list of possibilities for $f$
when $d(f)<2$ and $f$ is not a monomial.
Writing
$$
f(x,y) \ = \ a x^{\alpha}y^{\beta} \prod_{j=1}^M (y^t - \zeta_j x^r)^{n_j}
$$
as in \eqref{f} of Section \ref{prelim},
then $d(f)<2$ implies
\begin{equation}\label{restriction}
d(f) \ = \ \frac{t\alpha + r \beta + r t n}{r+t} \ < \ 2 \ \ \  {\rm or} \
\ \ t\alpha  +  r\beta + rtn \ \le 2r + 2t -1
\end{equation}
and this restricts the size of $n = \sum_{j\ge 1} n_j$, the total
number  of {\it nonzero} roots counted with multiplicities; we
necessarily have $1\le n \le 3$.

We enumerate the cases by the possible values of $n$,
starting with $n=3$. In this case we see from
\eqref{restriction} that necessarily $\alpha = \beta = 0$
and $t = r = 1$. This leads to the only possibilities
for $f$ being
\begin{equation}\label{n3}
f(x,y) \ = \ a (y - \eta x) (y - \zeta x) (y - \zeta^{*} x)
\end{equation}
where $\eta \in K$ is nonzero and either $\zeta$ and $\zeta^{*}$ are conjugate
elements of degree 2 over $K$ or both $\zeta$ and $\zeta^{*}$ are
elements of $K$.

Next we turn to the case $n=2$. In this case we see from
\eqref{restriction} that necessarily $t=1$ and
$0\le \alpha, \beta \le 1$ with at least
one equal to zero. This leads
to the only possibilities being
\begin{equation}\label{n2}
f(x,y) \ = a x^{\alpha} y^{\beta} (y - \zeta x^r)(y - \zeta^{*} x^r)
\end{equation}
with the above restriction on $\alpha,\beta$
and either $\zeta, \zeta^{*}$ are conjugate elements
of degree 2 over $K$ or the roots
$\zeta$ and $\zeta^{*}$ both
belong to $K$.
Finally we turn to the case $n=1$ where we have a single
nonzero root $\zeta$ lying in $K$ and so $f$ must be of the form
\begin{equation}\label{n1}
f(x,y) \ = \ a x^{\alpha} y^{\beta} (y^t - \zeta x^r).
\end{equation}
From \eqref{restriction} we see that $1\le t\le 3$
and $0\le \alpha,\beta \le 1$; if $t=3$,
then necessarily $r=4$ or $5$ and $\alpha=\beta=0$.
If $t=2$, then either $\alpha$ or $\beta$ (or both) is zero.
Furthermore when $t=2$, if
$\beta \not= 0$, then necessarily $\beta = 1$
and $r=3$.

We treat each case above separately. When $f$ is of the
form \eqref{n3}, then it is nondegenerate with respect
to its Newton diagram if the roots $\zeta, \zeta^{*}$
form a conjugate pair of degree 2 over $K$ or they
are distinct roots in $K$. In these cases $d(f) = 3/2 = h(f)$,
$m_K(f) = 1$
and the estimates \eqref{main-sum-est-abstract}
and \eqref{abstract-sum-infinite} follow
from the results in \cite{DS} and \cite{C-2}.
Strictly speaking the lower bound \eqref{abstract-sum-infinite}
is only shown in \cite{DS} in the setting of the
integers ${\Bbb Z}$. However since $m_K(f) = 1 < 3/2 = d(f)$,
the lower bound \eqref{abstract-sum-infinite} has already
been established in Section \ref{infinite}.

The remaining cases for \eqref{n3} are
when the roots $\zeta =\zeta^{*}$
coincide and lie in $K$. In this case
$d(f) = 3/2 < 2 \le m_K(f) = h(f)$ and
so we need to establish both bounds,
\eqref{main-sum-est-abstract} and \eqref{abstract-sum-infinite}.
We first consider
$$
f(x,y) \ = \ a (y-\eta x)(y - \zeta x)^2
$$
where $\zeta \not= \eta$; here $d(f)= 3/2$
and $m_K(f) =  h(f) = 2$. In this case we
make the change of variables $z = y - \zeta x$ in the $y$
integral \eqref{osc-S-chi}
representing the sum ${\mathcal S}_{\chi}$ and write
$$
{\mathcal S}_{\chi}(f;{\mathfrak p}^s) \ = \
\int\limits_{|x|\le 1,} \ \int\limits_{|z+\zeta x|\le 1}
\psi(\pi^{-s} a z^2(z-\zeta^{\prime} x)) d\mu(z)d\mu(x)
$$
$$
\ = \
\int\limits_{|z|\le 1} \psi(\pi^{-s} a z^3) d\mu(z)
\int\limits_{|x|\le 1} \psi(-\pi^{-s}\zeta^{\prime} z^2 x)
d\mu(x)
$$
where $\zeta^{\prime}  = \eta - \zeta$. Recall that
$|\zeta| = |\eta-\zeta|=1$ when ${\mathfrak p} \notin
{\mathcal P}(f)$ so that when $|z|\le 1$, $|x+\zeta^{-1} z| \le 1$
if and only if $|x|\le 1$.
The $x$ integral can be evaluated leading to the identity
${\mathcal S}_{\chi}(f;{\mathfrak p}^s) = q^{-s/2}$ if
$s$ is even and ${\mathcal S}_{\chi}(f;{\mathfrak p}^s) =
q^{-s/2} q^{-1/2}$ if $s$ is odd. From these identities,
we see that \eqref{main-sum-est-abstract} and
\eqref{abstract-sum-infinite}
hold.

The last case for \eqref{n3} is when $f(x,y) = a(y - \zeta x)^3$
and here $d(f) = 3/2 < m_K(f) = h(f) = 3$. A straightforward
computation shows ${\mathcal S}_{\chi}(f;{\mathfrak p}^s) = q^{-s/3}$
if $s\equiv 0 $ mod $3$,
${\mathcal S}_{\chi}(f;{\mathfrak p}^s) =
q^{-s/3}q^{-1/3}$ if $s\equiv 2$ mod $3$ and
$$
{\mathcal S}_{\chi}(f;{\mathfrak p}^s) \ = \ q^{-s/3}q^{1/3}\,
{\mathcal S}_{\chi}(f;{\mathfrak p})
$$
if $s\equiv 1$ mod $3$. Since
$$
{\mathcal S}_{\chi}(f;{\mathfrak p}) \ = \
\int_{|x|\le 1} \psi(\pi^{-1} a x^3) d\mu(x),
$$
we have $|{\mathcal S}_{\chi}(f;{\mathfrak p})| \le C q^{-1/2}$
from property $(C3)$ in Section
\ref{homogeneous} for character sums and so the estimate
\eqref{main-sum-est-abstract} holds in this case. Considering
the sequence $s=3k$ shows that \eqref{abstract-sum-infinite}
also holds in this case.

We now turn to those $f$ in \eqref{n2} where $f$
is nondegenerate with respect to its Newton diagram
unless the roots $\zeta, \zeta^{*}$ coincide
and lie in $K$. In the nondegenerate case,
the estimate \eqref{main-sum-est-abstract} follows
again from \cite{DS} or \cite{C-2}. For the lower
bound \eqref{abstract-sum-infinite}, we note that
$m_K(f)\le 1 <
2r/(r + 1) \le d(f)$ unless $\alpha = \beta = 0$ and $r=1$
in which case $m_K(f) = 0 < 1 = d(f)$. In either case
$m_K(f) < d(f) <2$ and so \eqref{abstract-sum-infinite}
follows from Section \ref{infinite}.

When $\zeta=\zeta^{*} \in K$,
we have $f(x,y) = a x^{\alpha}y^{\beta}
(y - \zeta x^r)^2$ where $0\le \alpha, \beta \le 1$,
not both of which are $1$. If $\alpha = \beta = 0$, a  simple
change of variables shows
$$
{\mathcal S}_{\chi}(f;{\mathfrak p}^s) \ = \
\int_{|x|\le 1} \psi(\pi^{-s} a  x^2) d\mu(x)
$$
and the integral above has modulus equal to $q^{-s/2}$
(we are assuming the characteristic of $K$, if positive,
is greater than $2$ in this case and so the element
$2 = 2\cdot {\bf 1}$ is nonzero; furthermore, we ensure that
the nonzero $2 = 2\cdot {\bf 1}$ lies in our collection
of algebraic elements ${\mathcal A}$
so that $|2| = 1$ whenever ${\mathfrak p} \notin
{\mathcal P}(f)$). If either $\alpha$ or
$\beta$ equals to $1$, then a computation similar to the
ones performed above shows that ${\mathcal S}(f;{\mathfrak p}^s)
= q^{-s/2}$ is $s$ is even and equal to $q^{-s/2}q^{-1/2}$
when $s$ is odd. In each case we see that both
\eqref{main-sum-est-abstract} and \eqref{abstract-sum-infinite}
hold.

Every $f$ arising in \eqref{n1} is nondegenerate
with respect to its Newton diagram and so \cite{DS}
or \cite{C-2} shows that \eqref{main-sum-est-abstract}
holds
for each such $f$ except $f(x,y) = a y (y-\zeta x)$ or
$f(x,y) = a x (y - \zeta x^r)$ where the bisectrix
passes through the vertex $(1,1)$ of the Newton diagram.
We treated these special cases at the beginning
of this subsection, noting the linear
factor $s$ does not arise in the estimates as predicted
by Theorem \ref{main-abstract}.
As for the lower bound \eqref{abstract-sum-infinite},
we need only verify this bound when $d(f)\le m_K(f) \le 2$
and $d(f) < 2$; the remaining cases have been
treated in Section \ref{infinite}. One easily checks
that $f(x,y) = a y^2 (y - \zeta x)$ with $\zeta \in K\setminus\{0\}$
is the only example in \eqref{n1} satisfying these conditions.
In this case the oscillatory integral in \eqref{osc-S-chi} becomes
$$
\int_{|y|\le 1} d\mu(y) \int_{|x|\le 1}
\psi(\pi^{-s} a y^2 (y - \zeta x)) d\mu(x)
=
\int_{|y|\le 1} d\mu(y) \int_{|z|\le 1}
\psi(\pi^{-s} a y^2 z) d\mu(z)
$$
using the change of variables $z = y - \zeta x$
in the $x$ integral and noting $|\zeta| = 1$
whenever ${\mathfrak p} \notin {\mathcal P}(f)$.
The integral on the right hand side is equal
to $q^{-s/2}$ if $s\equiv 0$ mod 2 and
$q^{-s/2}q^{-1/2}$ if $s\equiv 1$ mod 2. Since
$d(f) = 3/2 < m_K(f) = 2 = h(f)$, we see that
\eqref{abstract-sum-infinite} holds when $s\equiv 0$
mod 2 in this case.

\section{Appendix: the case when $f(x,y) = a x^{\alpha}y^{\beta}$ is a monomial}\label{monomial}

For completeness we treat
the simple case when $f(x,y) = a x^{\alpha} y^{\beta}$
is a single monomial and give a quick
analysis of the integrals
$$
I_{\alpha,\beta}  \ := \ \int\!\!\!\int_{{\bar{\mathfrak o}}\times
{\bar{\mathfrak o}}} {\mathcal C}(\pi^{-s} a x^{\alpha} y^{\beta}) d\mu(x)
d\mu(y)
$$
where ${\mathcal C}$ is either $\psi$,
the additive character on ${\bar{\mathfrak o}}$ so that
$I_{\alpha,\beta} = {\mathcal S}_{\chi}(f; {\mathfrak p}^s)$ is a character sum over the factor ring
${\mathfrak o}/{\mathfrak p}^s$, or it is equal to
the indicator function ${\bf 1}_{{\bar{\mathfrak o}}}$ of
${\bar{\mathfrak o}}$ so that $I_{\alpha,\beta} = {\mathcal N}(f;{\mathfrak p}^s)$ counts the
number of polynomial congruences $f(x,y) \equiv 0$
mod ${\mathfrak p}^s$. Since $f$ is quasi-homogeneous,
at least one exponent
$\alpha$ or $\beta$ is nonzero.
Also $|a| = |a|_{\mathfrak p} = 1$
for ${\mathfrak p} \notin {\mathcal P}(f)$.
In this case the height $h(f)$ is equal to $\max(\alpha,\beta)$
and $\nu(f) = 1$ or $0$ depending on whether $\alpha=\beta$
or not, respectively. The same is true for $i(f)$ except
when $\alpha = \beta = 1$ we have $i(f) = 0$ (in this case,
$\nu(f) = 1$).

When $f$ is linear, that is, when $f(x,y) = ax$
or $f(x,y) = a y$, we have $h(f) = 1$,
$i(f) = \nu(f) =0$, ${\mathcal S}_{\chi}(f;\pi^s) = 0$
and ${\mathcal N}(f;\pi^s) = q^{-s}$ so that
the bounds \eqref{main-sum-est-abstract},
\eqref{main-poly-est-abstract} and \eqref{abstract-poly-infinite}
trivially hold in this case (recall that the lower
bound \eqref{abstract-sum-infinite} holds in all
cases except when $f$ is linear in which case it
cannot possibly hold).

When $f(x,y) = a x y$, we have $h(f) = 1$, $i(f) = 0$
and $\nu(f) = 1$. In this case, ${\mathcal S}_{\chi}(f;\pi^s) = q^{-s}$
and ${\mathcal N}(f;\pi^s) = (1 - q^{-1}) s q^{-s} + q^{-s}$;
see below for this computation. Hence the estimates
in Theorem \ref{main-abstract} all hold in this case.

Therefore we may assume that $h(f) = \max(\alpha,\beta)\ge 2$.
Without loss of generality, suppose that $\alpha \le \beta$.
We decompose $I_{\alpha,\beta} =$
$$
\int\!\!\!\int_{{\bar{\mathfrak o}}\times
{\bar{\mathfrak o}}} {\mathcal C}(\pi^{-s} a x^{\alpha}y^{\beta}) \, d\mu(x) d\mu(y) =
\sum_{k \ge 0} q^{-k} \int_{|y|=1} d\mu(y)
\int_{|x|\le 1}
{\mathcal C}(\pi^{-s+\beta k} [a y^{\beta}] x^{\alpha})\, d\mu(x)
$$
$$
= (1-q^{-1}) \sum_{\beta k \ge s} q^{-k}
\ \  + \ \  \sum_{\beta k \le s-1} q^{-k} \int_{|y|=1} \int_{|x|\le 1}
{\mathcal C}(\pi^{-(s-\beta k)} [a y^{\beta}] x^{\alpha}) d\mu(x)
d\mu(y).
$$
If $\alpha = \beta$, then we can make the change of
variables $z = y x$ in the $x$ integral so that
$$
I_{\beta,\beta} = (1 - q^{-1}) \bigl[\sum_{\beta k \ge s} q^{-k} \ + \
\sum_{\beta k \le s-1}
q^{-k} \int_{|z|\le 1} {\mathcal C}(\pi^{-(s-\beta k)} a z^{\beta}) d\mu(z)\bigr]
$$
and the $z$ integral vanishes when $\beta = 1$ and
${\mathcal C} = \psi$. Furthermore the $z$ integral is equal to $q^{-(s-k)}$ when $\beta = 1$ and ${\mathcal C} =
{\bf 1}_{\bar{\mathfrak o}}$. This gives that values
of ${\mathcal S}_{\chi}(f;\pi^s)$ and ${\mathcal N}(f;\pi^s)$
for $f(x,y) = a x y$ mentioned above.
For $\beta \ge 2$, the $z$ integral is equal to $q^{-s/\beta}q^{k}$
when $s\equiv 0$  mod $\beta$ and this holds for
both polynomial congruences, ${\mathcal C} = {\bf 1}_{\bar{\mathfrak o}}$, and character sums, ${\mathcal C} = \psi$.
This shows that both \eqref{abstract-poly-infinite}
and \eqref{abstract-sum-infinite} hold for $s\equiv 0$
mod $\beta$ when $\alpha = \beta \ge 2$.
For general $s\ge 1$, we have the upper bound
$$
\bigl|
\int_{|z|\le 1} {\mathcal C}(\pi^{-(s-\beta k)} a z^{\beta}) d\mu(z)
\bigr| \ \le \ q^{-s/\beta} q^{k}
$$
for the $z$ integral, valid
for both ${\mathcal C} = {\bf 1}_{\bar{\mathfrak o}}$
or ${\mathcal C} = \psi$. This gives the upper bounds
in \eqref{main-sum-est-abstract} and \eqref{main-poly-est-abstract}
when $\alpha = \beta \ge 2$. Finally we note that when
${\mathcal C} = {\bf 1}_{\bar{\mathfrak o}}$, the $z$
integral has the lower bound $q^{-s/\beta}q^k q^{-1}$
for general $s\ge 1$ and this gives the lower bound
in \eqref{main-poly-est-abstract} in this case.

Finally we turn to treat the case $h(f) = \max(\alpha,\beta)\ge 2$
and $\alpha < \beta$. We will assume $\alpha \ge 1$;
the case $\alpha = 0$ is easier. As in the case for $\alpha = \beta$,
we have the upper bound
$$
\Bigl|
\int_{|y|=1} \int_{|x|\le 1}
{\mathcal C}(\pi^{-(s-\beta k)} [a y^{\beta}] x^{\alpha}) d\mu(x)
d\mu(y) \Bigr| \ \le \ C
q^{-(s-\beta k)/\alpha}
$$
and this leads to the upper bounds in \eqref{main-sum-est-abstract}
and \eqref{main-poly-est-abstract} for the case $\alpha < \beta$.
Furthermore, if $s = m_{*} \beta$ for some $m_{*}\ge 1$,
this upper bound implies
$$
\Bigl|
\sum_{k \le m_{*}-1} q^{-k} \int_{|y|=1} \int_{|x|\le 1}
{\mathcal C}(\pi^{-(s-\beta k)} [a y^{\beta}] x^{\alpha}) d\mu(x)
d\mu(y)
\Bigr| \ \le \ C q^{-s/\beta} q^{-([\beta/\alpha] - 1)}
$$
and therefore, when $s\equiv 0$ mod $\beta$,
 $|I_{\alpha,\beta}| \ge c q^{-s/\beta}$
if $q$ is large enough and this gives the lower
bounds \eqref{abstract-poly-infinite} and \eqref{abstract-sum-infinite}
in this case. Finally we observe that when
${\mathcal C} = {\bf 1}_{\bar{\mathfrak o}}$,
the lower bound
$$
\int_{|y|=1} \int_{|x|\le 1}
{\mathcal C}(\pi^{-(s-\beta k)} [a y^{\beta}] x^{\alpha}) d\mu(x)
d\mu(y) \ \ge \
q^{-(s-\beta k)/\alpha}q^{-1}
$$
leads to the lower bound in \eqref{main-poly-est-abstract}
for the case $\alpha < \beta$ and $h(f) = \max(\alpha,\beta)\ge 2$.

This completes our analysis for the monomial
case $f(x,y) = a x^{\alpha} y^{\beta}$ and hence
this completes the proof of Theorem \ref{main-abstract}.


\end{document}